\documentclass{scrartcl}
\usepackage{amssymb,amsmath,amsthm}
\usepackage[pdftex]{color}
\usepackage{latexsym}
\usepackage[pdftex]{graphicx}
\usepackage{lmodern} 
\usepackage[utf8]{inputenc} 
\usepackage{cite}
\usepackage{nicefrac}
\usepackage{esint}
\usepackage[]{hyphenat}

\usepackage[pdfauthor={Kathrin Stollenwerk}, pdftitle={Optimal Shape of a Domain which minimizes the first Buckling EIgenvalue},,pdfproducer={LaTeX with hyperref package},pdfpagelayout=SinglePage,bookmarksnumbered,pdftex=true,breaklinks=true,bookmarks=true,linktocpage=true,pdfpagelabels=true,plainpages=false,pagebackref=false]{hyperref} 
\hypersetup{colorlinks=true,linkcolor=blue,citecolor=blue,urlcolor=blue,filecolor=blue}
\title{Optimal Shape of a  Domain which minimizes the first Buckling Eigenvalue}
\author{Kathrin Stollenwerk}
\date{}

    \newcommand{\bilap}[1]{\Delta^{\!2}\hspace{-0.25mm}#1 }
    \newcommand{\grad}[1]{\nabla \! #1}
    \newcommand{\lap}[1]{\Delta \hspace{-0.15mm} #1 }
    \newcommand{\aq}{\Leftrightarrow} 
    \renewcommand{\epsilon}{\varepsilon}
    \newcommand{\eps}{\varepsilon}
     \renewcommand{\theta}{\vartheta}   
    \renewcommand{\phi}{\varphi}   
    \newcommand{\norm}[1]{\|#1\|}
    \newcommand{\abs}[1]{\left|#1\right|}
    \newcommand{\J}{\mathcal{J}}
     \newcommand{\Je}{\mathcal{J}_\epsilon} 
 \newcommand{\Ie}{\mathcal{I}_\epsilon} 
    \newcommand{\C}{{\mathcal{C}^\eps\!\!}}        
    
    \renewcommand{\L}{\Lambda}  
    \newcommand{\Le}{\L_\epsilon}   
     \newcommand{\R}{\mathbb{R}}   
    \newcommand{\N}{\mathbb{N}}    		
  \newcommand{\dist}{\operatorname{dist}}
    \newcommand{\ue}{u_\epsilon}	
    \newcommand{\fe}{f_\epsilon}
    \newcommand{\interior}{\operatorname{int}}
  \newcommand{\U}{\mathcal{U}}
    \newcommand{\Ue}{\U_\epsilon}
  \newcommand{\E}{\mathcal{F}_\eps}
  \newcommand{\W}{\mathcal{W}}
 \newcommand{\Geo}{\Gamma^0_\eps}
 \newcommand{\Ge}{\Gamma^1_\eps}
\renewcommand{\O}{\mathcal{O}}
\newtheorem{theorem}{Theorem}

\newtheorem{definition}{Definition}
\newtheorem{corollary}{Corollary}
\newtheorem{remark}{Remark}
\newtheorem{lemma}{Lemma}

\begin{document}
\maketitle
\begin{abstract}
In this paper we prove the existence of an optimal domain which minimizes the buckling load of a clamped plate among all bounded domains with given measure. Instead of treating this variational problem with a volume constraint, we introduce a problem without any constraints, but with a penalty term. We concentrate on the minimizing function and prove that it has Lipschitz continuous first derivatives. Furthermore, we show that  the penalized problem and the original problem can be treated as equivalent. Finally, we establish some qualitative properties of the free boundary.
%\subclass{49K20 \and 35J30}
%\keywords{Calculus of variations \and buckling load \and clamped plate}
\end{abstract}
%%%% INTRODUCTION
\section{Introduction}
The question, which domain minimizes the first eigenvalue of the Laplacian with Dirichlet boundary conditions, is probably one of the most famous questions in shape optimization. In 1877, Lord Rayleigh  claimed that among all plane domains with the same area the disk is the optimal domain \cite{rayleigh}. In the 1920s, G.~Faber  and E.~Krahn simultaneously, but independently, proved Lord Rayleigh's conjecture \cite{faber,krahn}.
\newline
The present work is concerned with an apparently analogue question, namely which domain of given measure minimizes the buckling load of a clamped plate? There exists a conjecture concerning the buckling load, which is analogue to Rayleigh's conjecture. In 1951, G. Polya and G. Szegö  claimed 
that the ball minimizes the buckling load among all open sets of given measure \cite{PolyaSzego}. 
Thereby the buckling load of a domain $\Omega \subset \R^n$, $n\geq 2$, is defined as
\[
     \L(\Omega) := \min_{\substack{ v \in H^{2,2}_0(\Omega) \\ v \not\equiv 0 }}\frac{\int\limits_{\Omega}\abs{\lap v}^2 dx}{\int\limits_{\Omega}\abs{\grad v}^2dx}.
\]
For each $\Omega$ there exists an $u\in H^{2,2}_0(\Omega)$ for which the infimum for $\L(\Omega)$ is attained. The function $u$ satisfies
\[
  \bilap u + \L(\Omega) \lap u =0
\]
in $\Omega$. Thus, we call $\L(\Omega)$ the first buckling eigenvalue of $\Omega$.
\newline
The Polya-Szeg\"o conjecture is still not proven. However, some partial results are known. Assuming that the first eigenfunction does not change its sign, G.~Szegö  gave a proof \cite{szego50,PolyaSzego}. Though, in general the eigenfunction does not satisfy the assumed property. 
Considering the two-dimensional case, two uniqueness results are known. Assuming that a smooth and simply connected optimal domain exists,  H.\,F.~Weinberger and B.~Willms (see \cite{Wi95}) were able to prove the Polya-Szegö conjecture. Performing the shape derivative of the optimal domain, they obtained a further boundary condition for the eigenfunction. Denoting the eigenfunction by $u$, they found that $\lap u+\L u = const.$ in the optimal domain. Subsequently, applying estimates between the first buckling eigenvalue and higher Dirichlet-Laplace eigenvalues they could prove that the optimal domain is a disk.  Secondly, it is possible to adopt the proof of E.~Mohr, who showed that under the previous assumptions the disk minimizes the first eigenvalue of a clamped plate \cite{mohr}, to the buckling of a clamped plate \cite{Kn08}. This approach uses the second domain derivative of the optimal domain.
\newline
In 2003, M.\,S.~Ashbaugh and D.~Bucur in \cite{BucurAshbaugh} proved the existence of an optimal domain among all simply connected domains of given measure in two dimensions. They did not gain any result regarding the regularity of the optimal domain, but they outlined possible ways of applying the Weinberger-Willms idea without a priori assuming  the regularity of the optimal domain. 
\newline
The previous mentioned articles focus on the optimal domain and  examine the eigenfunction just marginally. In this work, we choose an opposed strategy and concentrate ourselves on the eigenfunction.
In this way, we prove the existence of an optimal domain among all open sets of given measure which are contained in a large ball $B \subset \R^n$ ($n\in\{2,3\}$). Thus, we avoid the difficulties, which apprear considering subsets of $\R^n$ instead of subsets of $B$. 
Particularly, the existence of an optimal domain now follows from the direct method in the calculus of variation. In contrast to M.\,S. Ashbaugh and D. Bucur in \cite{BucurAshbaugh}, we do not require any concentration compactness methods.
In order to obtain an optimal domain, which fulfils the volume condition, we solve a penalized variational problem. In this way, the volume condition is not a side condition anymore and we obtain a variational problem without any constraints.
We will prove the existence of a solution for the penalized problem and show that this solution solves the original problem if the penalized problem satisfies a certain condition. 
Furthermore, we will obtain that the first order derivatives of each solution of the penalized problem are Lipschitz continuous. We will prove that the first order derivatives of the solutions do not degenerate along the free boundary. Consequently, we can establish a lower bound on the free boundary's density and derive that the free boundary is a nullset with respect to the $n$-dimensional Lebesgue measure. Moreover, the lower bound on the density allows us to deduce some results regarding the shape of the optimal domain.
\newline
In the sequel, we analyze the following minimizing problems. 
Let $B \subset \R^n$ be a ball with large volume, i.\,e. \!$\abs{B}>>1$. Then for $v \in H^{2,2}_0(B)$ the functional $\J : H^{2,2}_0(B)\to \R$ is defined by
\begin{equation}\notag
 \J(v) := \frac{\int\limits_B \abs{\lap v}^2dx}{\int\limits_{B}\abs{\grad v}^2dx}\,.
\end{equation}
In addition, for a function $v \in H^{2,2}_0(B)$ we set 
\begin{equation}\label{eq:O}
  \O(v) := \{x \in B: v(x)\neq 0\}.
\end{equation}
Now we fix an $\omega_0\in\R$ with $0<\omega_0<< \abs{B}$. This quantity $ \omega_0$ is the intended volume the optimal domain should attain. 
Hence, the question which set among all open subsets of $B$ with given measure $\omega_0$ minimizes the first buckling eigenvalue is equivalent to the following variational problem:
\vspace{-.6cm}\begin{center}
\begin{equation*}
 \boxed{ \begin{gathered}
 \text{Find a function } u\in H^{2,2}_0(B) \text{ with} \abs{\O(u)} = \omega_0 \text{ such that }  \\
  \J(u) = \min_{\substack{v \in H^{2,2}_0(B)\\\abs{\O(v)}=\omega_0}} \J(v).
 \end{gathered} }
\tag{$P$}\label{P}\end{equation*}
\end{center}
This is the actual problem we will solve. 
Assuming that $u \in H^{2,2}_0(B)$ solves the problem~\eqref{P}, the optimal set is $\O(u)$.
 However, the volume condition $\abs{\O(u)}=\omega_0$ causes several difficulties while discussing the problem \eqref{P}. Whenever we perturb a function $u \in H^{2,2}_0(B)$ with  $\abs{\O(u)}=\omega_0$, we have to guarantee that the perturbed function satisfies the volume constraint as well. To avoid this difficulty, we follow an idea of H.~W.~Alt and L.~A.~Caffarelli in \cite{AltCaf81} and 
consider a penalized problem. In this way,  non-volume preserving perturbations are allowed. For this purpose, we define for $\eps>0$ the function $\fe : \R \to \R$ by
\begin{equation}\label{fe}
  \fe(s) := \begin{cases}
             0, & s\leq \omega_0 \\
             \frac{1}{\epsilon}(s-\omega_0), &s\geq \omega_0
            \end{cases}.
\end{equation}
Now we set for $v\in H^{2,2}_0(B)$
\begin{equation}\label{eq:def_Je}
 \Je(v) := \J(v) + \fe(\abs{\O(u)}) \,.
\end{equation}
The additional term $\fe$ penalizes the functional if the $n$-dimensional Lebesgue measure of $\O(u)$ gets larger than $\omega_0$. 
Thus, we may omit the side condition '$\abs{\O(u)}=\omega_0$' in the problem \eqref{P} and obtain  
 the following new variational problem, in which no constraints occur: 
\vspace{-.6cm}\begin{center}
\begin{equation*}
 \boxed{ \begin{gathered}
 \text{Find a function } \ue\in H^{2,2}_0(B)  \text{ such that }  \\
  \Je(u) = \min_{v \in H^{2,2}_0(B)} \Je(v).
 \end{gathered} }
\tag{$P_\eps$}\label{Pe}\end{equation*}
\end{center}
Handling this problem is much more comfortable than the problem~\eqref{P}. Indeed, in the sequel we will often take advantage of the opportunity to perform non-volume preserving perturbations of the eigenfunction. 
%%%%
%%%% SECTION 2
\section{Existence of a Solution of the Penalized Problem}
To begin with, we prove the existence of solutions for the penalized problem \eqref{Pe} and establish a first regularity result for the minimizers. 
We obtain the existence of a minimizer $\ue\in H^{2,2}_0(B)$ for the functional $\Je$ for every $\eps>0$ and show that each  $\ue$ is a solution of the buckled plate equation in the set $\O(u)$. Moreover, we  obtain that the $n$-dimensional Lebesgue measure of $\O(u)$ cannot be smaller than the intended volume $\omega_0$.
In addition, we detect the Hölder continuity of the first order derivatives of the solutions of the penalized problem \eqref{Pe}.

Using the direct method in the calculus of variation, we prove the existence of a solution of the penalized problem \eqref{Pe}. 

\begin{definition}
If a function $w_\eps$ in $H^{2,2}_0(B)$ satisfies
 \[
   \Je(w_\eps) \; \leq \; \Je(v)
 \]
 for all $v\in H^{2,2}_0(B)$,  we call $w_\eps$ a minimizer of the functional $\Je$ in $H^{2,2}_0(B)$ or a solution of the problem  \eqref{Pe}. 
\end{definition}

\begin{theorem}\label{existence_sobo}
  For every $\eps>0$ there exists a minimizer $u_\epsilon \in H^{2,2}_0(B)$ of the functional $\Je$.
\end{theorem} 
\begin{proof}
Since the functional $\Je$ (as defined in \eqref{eq:def_Je}) is nonnegative, there exists a minimizing sequence $(u_k)_k \in H^{2,2}_0(B)$ with
\[
  \lim_{k\to\infty} \Je(u_k) \; = \; \inf_{v\in H^{2,2}_0(B)} \Je(v)\,.
\]
Without loss of generality, we assume that there exists a constant $C>0$ such that $\Je(u_k) \leq C$ for all $k \in \N$; otherwise we set $\Je(u_k) = \infty$. Thus, we are able to normalize the sequence $(u_k)_k$ such that $\norm{\grad u_k}_{L^2(B)}=1$ for all $k\in\N$.
This normalization implies 
 \[
  \norm{u_k}^2_{H^{2,2}_0(B)}\,=\,   \int\limits_B \abs{D^2u_k}^2\,dx\,=\,\int\limits_B \abs{\lap{u_k}}^2\,dx \, \leq \Je(u_k) \leq C\,.
 \]  
Thus, there exists a subsequence $(u_k)_k$ which converges weakly to an $u_\epsilon$ in $H^{2,2}_0(B)$\,. We observe that $\norm{\grad \ue}_{L^2(B)}=1$.
The lower semicontinuity of the $H^{2,2}_0(B)$ norm implies
 \[
    \int\limits_B \abs{\lap{u_\epsilon}}^2dx \, =\, \int\limits_B \abs{D^2u_\epsilon}^2dx \, \leq \liminf_{k\to \infty} \int\limits_B \abs{D^2u_{k}}^2dx \, = \, 
    \liminf_{k\to \infty}\int\limits_B \abs{\lap{u_{k}}}^2dx\,.
 \] 
It remains to prove the lower semicontinuity of the penalization term $\fe$ with respect to the weak convergence in $H^{2,2}_0(B)$. Since $\fe$ is nondecreasing, it is sufficient to show that 
\begin{equation} \label{eq:vol_est}
   \abs{\O(\ue)} \leq \liminf_{k\to\infty}\abs{\O(u_k)}.
\end{equation}
where $\O(u_{k})$ and $\O(\ue)$ are defined as in \eqref{eq:O}.
For this purpose, note that due to the theorem of Banach-Alaoglu, there exists a function $\beta \in L^\infty(B)$ with $0\leq\beta\leq1$ such that  (at least for a subsequence of $(u_k)_k$) there holds
\[
   \lim_{k\to\infty}\int\limits_B\chi_{\O(u_k)}\phi \,dx = \int\limits_B\beta\phi\,dx \quad \forall\phi \in L^1(B).
\] 
Consequently, we obtain
\begin{align*}
 0 &= \lim_{k\to\infty}\left(  \int\limits_B u_k^+(1-\chi_{\O(u_k)})\,dx + \int\limits_B u_k^-(1-\chi_{\O(u_k)})\,dx  \right) \\
 &= \int\limits_B \ue^+(1-\beta)\,dx + \int\limits_B \ue^-(1-\beta)\,dx,
\end{align*}
where $v^+:= \max\{v,0\}$ and $v^-:= \max\{-v,0\}$. Since both summands are nonnegative, this implies $\beta = 1$ almost everywhere in $\O(\ue)$. Hence,
\[
 \abs{\O(\ue)} = \int\limits_{\O(\ue)}1\,dx \leq \int\limits_B\beta\,dx = \liminf_{k\to\infty}\int\limits_B\chi_{\O(u_k)}\,dx = \abs{\O(u_k)}. 
\]
Finally, we find
\[
   \Je(u_\epsilon) \, \leq \,\liminf_{k \to\infty} \Je(u_k) \, = \, \inf_{v\in H^{2,2}_0(B)}\Je(v).
\] 
\end{proof}
Note that for $n\in \{2,3\}$ the set $\O(\ue)$ (defined as in \eqref{eq:O})
 is an open subset of $B$. Specifically, the normalization $\norm{\grad\ue}_{L^2(B)}=1$ provides that $\ue\not\equiv 0$ and so $\O(\ue)\neq\emptyset$.
\begin{remark}\label{rem:phases_touch}
We have to think of $\ue$ as of a function which changes its sign. Hence, the positve and the negative phase are not empty, i.e.
\[
   \{\ue>0\}\neq \emptyset \; \mbox{and} \; \{\ue<0\}\neq \emptyset.
\]
Note that, if the positive and negative phase have positive distance, we can substitute $\ue$ by $\abs{\ue}$. Thus, we may think of the minimizer $\ue$ as of a function, which does not change its sign. Then G.~Szegö showed that $\O(\ue)$ is a ball \cite{szego50,PolyaSzego}. \\
Therefore, in this work, we only consider the case that the two phases touch. 
\end{remark}
\vspace{.3cm} 
Note that we do not have any information whether the $n$-dimensional Lebesgue measure of $\partial\O(\ue) = \partial\{\ue>0\}\cup\partial\{\ue>0\}$ is zero or not. We cannot answer this question until Lemma~\ref{lem:nullset}. 
From now on, we set
\begin{equation}\label{eq:Le}
 \Le := \int\limits_B\abs{\lap\ue}^2dx = \J(\ue).
\end{equation}
Classical variational arguments show that each minimizer $\ue$ of the functional $\Je$ in $H^{2,2}_0(B)$  solves the buckled plate equation in $\O(\ue)$, i.\,e.
 \begin{equation}\label{eq:var_eq}
    \bilap\ue+\Le\lap\ue=0 \, \text{ in } \O(\ue) 
 \end{equation}
where $\Le$ is defined in \eqref{eq:Le}.

The next theorem shows that for every $\eps>0$ the $n$-dimensional Lebesgue measure of $\O(\ue)$ cannot fall below $\omega_0$. 
This result is independent of $\eps$. 
\begin{theorem}\label{supp_bigger}
 Suppose $\ue \in H^{2,2}_0(B)$ minimizes $\Je$. Then for each $\eps>0$ there holds
\[
 \abs{\O(\ue)} \geq \omega_0\,.
\]
\end{theorem}
\begin{proof}
 We assume that $ \abs{\O(\ue)}< \omega_0$ holds for at least one $\eps>0$.
Since $\ue$ is continuous, we can choose $x_0 \in \partial\O(\ue)\setminus\partial B$ such that $x_0$ is an accumulation point of  $\{x \in B : \ue(x)=0\}$. Moreover, let $r>0$ such that $B_r(x_0)\subset B$.
We define $\hat{v}$ as
\begin{equation}\notag
 \hat{v} := \begin{cases}
               \ue, & \text { in }  B\setminus B_{r}(x_0)\\
               v, & \text{ in } B_r(x_0)
            \end{cases}\,,
\end{equation}
where $v-\ue \in H^{2,2}_0(B_r(x_0))$ and  $v$ is a solution of 
\begin{equation}\notag
     \bilap v + \Le \lap v =0 \quad \text{ in } B_r(x_0) \,.
\end{equation}
 Obviously, there holds \, $\O(\hat{v})\subset\left(\O(\ue)\cup B_r(x_0)\right)$ \,
and our assumption implies $\abs{\O(\hat{v})}  \leq \omega_0$ if $r$ is chosen sufficiently small. Therefore, the penalization term $\fe$ (as defined in \eqref{fe}) fulfils $\fe(\abs{\O(\hat{v})})=0$ and $\fe(\abs{\O(\ue)})=0$, of course. The minimality of $\ue$ with respect to $\Je$ then leads to the following local inequality:
\begin{equation}\label{eq:bigger:local}
  \int\limits_{B_r(x_0)} \abs{\lap \ue}^2 - \abs{\lap v}^2dx \leq \Le \int\limits_{B_r(x_0)}\abs{\grad \ue}^2 - \abs{\grad v}^2dx\,. 
\end{equation}
Using integration by parts, Green's identity and the definition of $v$, we calculate
\begin{align*}
 &\int\limits_{B_r(x_0)}\hspace{-2mm}\abs{\lap \ue}^2\, - \abs{\lap v}^2dx  
= \; \int\limits_{B_r(x_0)}\hspace{-2mm}\abs{\lap (\ue-v)}^2 dx + 2\Le\int_{B_r(x_0)}\hspace{-2mm}\grad v.\grad(\ue-v)\,dx\,.
\end{align*}
Hence, inequality \eqref{eq:bigger:local} reads as
\[
 \int\limits_{B_r(x_0)}\abs{\lap (\ue-v)}^2dx \leq \Le \int\limits_{B_r(x_0)}\abs{\grad (\ue-v)}^2dx\,.
\]
Since $\ue-v \in H^{2,2}_0(B_r(x_0))$, we can apply Poincar\'{e}'s inequality. This yields
\begin{equation}\label{eq:bigger:contra}
 \int\limits_{B_r(x_0)}\abs{\lap (\ue-v)}^2dx \leq \Le\, r^2 \int\limits_{B_r(x_0)}\abs{\lap (\ue-v)}^2dx\,.
\end{equation}
Provided that the integral in \eqref{eq:bigger:contra} does not vanish, \eqref{eq:bigger:contra} is contradictory for $r$ small enough. 
Let us assume the integral in \eqref{eq:bigger:contra} vanishes. In this case, $\ue \equiv v$ in $\overline{B_r(x_0)}$. Furthermore, $v$ is analytic in $B_{\nicefrac{r}{2}}(x_0)$ because $v$ solves an elliptic equation. Consequently, $\ue$ is analytic in $B_{\nicefrac{r}{2}}(x_0)$, too. 
Since $x_0$ is an accumulation point of the zero set of $\ue$, the identity theorem for power series implies that $\ue\equiv 0$ in $B_{\nicefrac{r}{2}}(x_0)$. 
This is contradictory since $x_0 \in \partial\O(\ue)$. Thus, the claim is proven.
\end{proof}
\begin{remark}\label{rem:penal_term}
 The previous result is a consequence of the choice of the penalization term $\fe$. We have chosen a penalization term, which is monotone but not strictly monotone. Another possibility of defining the penalization term would be 
\[
   \hat{\fe}(s) := \begin{cases}
   -\eps(\omega_0 - s), & \text{if }  s\leq \omega_0 \\
  \frac{1}{\eps}(s-\omega_0), & \text{if } s> \omega_0
\end{cases}.
\] 
This penalization term was chosen by C.~Bandle and A.~Wagner in \cite{BaWa09}, e.g..
It rewards the functional if the support is smaller than $\omega_0$. This rewarding would annihilate our argumentation in the previous proof. 
 Indeed, if we chose $\hat{\fe}$ instead of $\fe$, the inequality (\ref{eq:bigger:contra}) would be
\begin{align*}
  (1-\L r^2)\int\limits_{B_r(x_0)}\abs{\lap(\ue-v)}^2dx  \leq  \eps \abs{B_r(x_0)\cap \{\ue=0\}}\int\limits_B\abs{\grad v}^2dx.
\end{align*}
Obviously, arguing like in the proof of Theorem~\ref{supp_bigger} does not work anymore. However, we would need an estimate in the form of 
\[
  \int\limits_{B_r(x_0)}\abs{\lap(\ue-v)}^2dx \geq C \abs{B_r(x_0)\cap \{\ue=0\}},
\]
where the constant $C$ is independet of $\eps$.
H.\,W.~Alt and L.\,A.~Caffarelli in \cite{AltCaf81} obtained an analogous estimate using comparison principles for solutions of second order partial differential equations. Since we do not possess any comparison or maximum principle, we are not able to adopt their approach. This is why, we have chosen the monotone, but not strictly monotone, penalization term $\fe$.
\end{remark}

\subsection{\texorpdfstring{$C^{1,\alpha}$ Regularity of the Minimizers $\ue$}{Hölder-Regularity}}\label{sec:c^1,alpha-reg}

Our next aim is to show the Hölder continuity of the first order derivatives of $\ue$. 
We choose an approach using Morrey's Dirichlet Growth Theorem, which has been used for similar problems (see \cite{BaWa07}, e.g.). 
For the proof of the following theorem we refer to \cite{morrey}.
\begin{theorem}[Morrey's Dirichlet Growth Theorem]\label{morrey_theorem} 
 Let $\phi \in H^{1,p}_{0}(B)$, $1\leq p\leq n$, $0 < \alpha \leq 1$ and suppose there exists a constant $M>0$, such that 
 \begin{equation}\label{eq:assumption_morrey}
   \int\limits_{B_r(x_0)\cap B} \abs{\grad{\phi(x)}}^pdx \; \leq \; M\,r^{n-p+\alpha p}
 \end{equation}
 for every $B_r(x_0)$ with $x_0 \in \overline{B}$. Then $\phi \in C^{0,\alpha}(\overline{B})$\,. 
\end{theorem}
We need to verify the assumptions of Theorem \ref{morrey_theorem} for the second order derivatives of $\ue$. Since we are not interested in any local results, we must allow to consider balls $B_R(x_0)$ leaving $B$. For this reason, we need the version of Morrey's Dirichlet Growth Theorem formulated above. Thus, we may consider $B_R(x_0)\cap B$ if $B_R(x_0)\not\subset B$.  Now let $x_0 \in \overline{B}$ and $R>0$. 
 We define the comparison function $\hat{v}$ by
\begin{equation}\label{eq:reg_1_11}
 \hat{v} := \begin{cases}
                \ue, & \text{in}\; B\setminus B_R(x_0) \\
		 v, & \text{in}\; B_R(x_0)\cap B
               \end{cases},
\end{equation}
where $v-\ue \in  H^{2,2}_0(B_R(x_0)\cap B)$ and $v$ is biharmonic in $B_R(x_0) \cap B$. We only consider balls $B_R(x_0)$ which intersect $\O(\ue)$ and satisfy $\O(\ue)\not\subset B_R(x_0)$; otherwise, the function $\hat{v}$ vanishes. Now suppose $0<r<R$. Then the estimate
\begin{equation}\label{eq:reg_1_1}
 \int\limits_{B_r(x_0)\cap B} \abs{D^2\ue}^2dx \leq 2\int\limits_{B_R(x_0)\cap B}\abs{D^2(v-\ue)}^2dx + 2 \int\limits_{B_r(x_0)\cap B}\abs{D^2 v}^2\,dx
\end{equation}
is obvious. The next lemma helps to estimate the last term in the above inequality.

\vspace{.3cm}
\begin{lemma}\label{reg_1_la_1}
Using the above notation, there exists a constant $C>0$ such that for each  $r <R$ the following estimate holds
 \[
   \int\limits_{B_r(x_0)\cap B}\abs{D^2 v}^2dx \leq C\,\left(\frac{r}{R}\right)^n \,\int\limits_{B_R(x_0)\cap B}\abs{D^2\ue}^2dx\,.
 \]
Thereby the constant $C$ is independent of $r,\,R$ and $x_0$\,.
\end{lemma}
\begin{proof}
 We first show that
\[
 \int\limits_{B_r(x_0)\cap B}\abs{D^2v}^2dx \leq C\left(\frac{r}{R}\right)^n \int\limits_{B_R(x_0)\cap B}\abs{D^2v}^2dx\,.
\]
Note that this claim is obvious if $\frac{R}{2}\leq r<R$. Assuming $r<\frac{R}{2}$, equation (3.2) in  the proof of Lemma 3.2 in \cite{GM1978} implies
\begin{align*}
 \int\limits_{B_r(x_0)\cap B}\abs{D^2v}^2dx 
&\leq  C \left(\frac{r}{R}\right)^n \int\limits_{B_R(x_0)\cap B}\abs{D^2v}^2dx\,.
\end{align*}
Since $v$ is the unique solution of
\[
\min\left\{  \int\limits_{B_R(x_0)\cap B}\abs{D^2w}^2dx: w - \ue\in  H^{2,2}_0(B_R(x_0)\cap B)\right\},
\]
the claim is proven.
\end{proof}
It remains to estimate the first integral on the right hand side of \eqref{eq:reg_1_1}. 
\begin{lemma}\label{reg_1_la_3}
Assume the same situation as in Lemma \ref{reg_1_la_1}. Then the estimate 
 \[
   \int\limits_{B_R(x_0)\cap B}\abs{D^2 (\ue-v)}^2dx \,\leq\, C(n,\omega_0,\Le)\,R^{n-2+2\alpha}
 \] 
holds for each $\alpha \in (0,1)$, where C is a positive constant depending only on $n,\,\omega_0$ and $\Le$.
\end{lemma}
\begin{proof}
We fix $R<1$ such that $\abs{B_R(x_0)}<\omega_0$. Due to Theorem \ref{supp_bigger} this implies $\abs{B_R(x_0)}<\abs{\O(\ue)}$ and the case $\O(\ue) \subset B_R(x_0)$ is excluded. As mentioned above, we only consider the case $B_R(x_0)\cap  \O(\ue)\neq \emptyset$; otherwise, $\ue-v$ vanishes in $B_R(x_0)$. We obtain the result by comparing the $\Je$-energies of $\ue$ and $\hat{v}$, where  $\hat{v}$ is defined as in \eqref{eq:reg_1_11}.  Yet, it may occur  that $\O(\hat{v}) \supset \O(\ue)$.  This inhibits a reasonable comparison of $\Je(\ue)$ and $\Je(\hat{v})$ because of the monotonicity of the penalization term.  We circumvent this problem by scaling $\hat{v}$. This step is not necessary if there holds $\O(\hat{v})\subset \O(\ue)$. In this case, we can compare $\Je(\ue)$ and $\Je(\hat{v})$ immediately. 
 Therefore, we now concentrate on the case $\O(\hat{v}) \supset \O(\ue)$.  
Let $B^\ast \subset B$ be a ball concentric to $B$.
We define $w(x) := \hat{v}(\mu x)$ for $x \in B^\ast$. Thereby, $\mu\geq1$ is chosen such that
$\abs{\O(w)}=\abs{\O(\ue)}$.  To be precise, $B$ is thought to be a ball with radius $R_0$ and centre $x_B$. We set $B^\ast = B_{R^\ast}(x_B) $, where
\[
  R^\ast := \frac{R_0}{\mu}  \qquad \text{and} \qquad \mu :=  \left(\frac{\abs{\O(\hat{v})}}{\abs{\O(\ue)}}\right)^\frac{1}{n} >1\,.
\]
Thus, $w \in  H^{2,2}_0(B^\ast)\subset H^{2,2}_0(B)$. 
 Furthermore, there holds \\ $\O(\hat{v})\subset \left(\O(\ue)\cup B_R(x_0)\right)$ and we estimate
\[
 1 < \mu \leq \left(1+\frac{\abs{B_R}}{\abs{\O(\ue)}}\right)^\frac{1}{n}.
\]
Using Taylor's expansion and Theorem \ref{supp_bigger}  yield
\begin{equation}\label{eq:reg_1_5}
 1-\mu^{-2} \leq C(n,\omega_0)R^n.
\end{equation}
The minimality of $\ue$ for $\Je$ in $H^{2,2}_0(B)$ now implies
\begin{equation}\label{eq:reg_1_6}\notag
\Le\int_B\abs{\grad w}^2 dy \leq \int\limits_B\abs{\lap w}^2dy \; \Leftrightarrow \; \Le\mu^{-2}\int\limits_{B}\abs{\grad \hat{v}}^2dx \leq \int\limits_{B}\abs{\lap\hat{v} }^2dx\,.
\end{equation}
Rearranging terms we obtain the local inequality
\[
 \int\limits_{B_R\cap B}\abs{\lap{\ue}}^2- \abs{\lap{v}}^2dx \leq \Le\left(1-\frac{1}{\mu^2}\right) + \frac{\Le}{\mu^2}\left[\,\,\int\limits_{B_R\cap B}\abs{\grad{\ue}}^2-\abs{\grad{v}}^2dx\right],
\]
where we denote $B_R= B_R(x_0)$ for simplicity. Since $v$ is biharmonic in $B_R\cap B$ and $v-\ue \in  H^{2,2}_0(B_R\cap B)$, we obtain 
\begin{equation}\label{eq:reg_1_8}
  \int\limits_{B_R\cap B}\abs{\lap{(\ue-v)}}^2dx \,\leq\, \Le\left(1-\frac{1}{\mu^2}\right) + \frac{\Le}{\mu^2}\,\int\limits_{B_R\cap B}\grad{(\ue-v)}.\grad{(\ue+v)}\,dx. 
\end{equation}
Integration by parts and Young's inequality imply
\[
   \int\limits_{B_R\cap B}\abs{\lap{(\ue-v)}}^2\,dx \; \leq \; 2\Le\left(1-\frac{1}{\mu^2}\right) + 2\Le^2 \int\limits_{B_R\cap B}\ue^2+v^2dx\,.
\]
Since we only consider $n\in\{2,3\}$, the classical Sobolev embedding theorems imply 
\[
 \abs{v(x)},\abs{\ue(x)}\leq C\norm{\ue}_{H^{2,2}_0(B)} = C(\Le)
\]
for each $x \in B_R(x_0)\cap B$. 
Hence,  estimate \eqref{eq:reg_1_5} yields
\begin{align*}
  \int\limits_{B_R(x_0)\cap B}\abs{\lap{(\ue-v)}}^2dx \,
  &\leq \, C(n,\omega_0,\Le)R^n.
  \end{align*}
Since $v-\ue \in  H^{2,2}_0(B_R(x_0)\cap B)$ and $R<1$ the claim is proven.
\end{proof}
The next lemma is the last technical tool, which is necessary to prove the $C^{1,\alpha}$ regularity of the minimizer. For the proof we refer to \cite{giaq_mult_int}, Chapter III.
\begin{lemma}\label{reg_1_la_2}
Let $\Phi$ be a nonnegative and nondecreasing function. Suppose that 
there exist positive constants $\gamma,\alpha,\kappa,\beta$ , $\beta<\alpha$, such that 
for all $0 \leq r \leq R \leq R_0$ 
 \[
   \Phi(r) \leq \gamma\,\left[\left(\frac{r}{R}\right)^\alpha + \delta\right]\,\Phi(R) + \kappa\,R^\beta.
 \]
Then there exists a constant $\delta_0 =\delta_0(\gamma,\alpha,\beta)$ such that if $\delta<\delta_0$, for all $r<R\leq R_0$ we have
 \[
   \Phi(r) \leq c\,\left(\frac{r}{R}\right)^\beta\,[\Phi(R) + \kappa\,R^\beta] \,,
 \]
 where $c$ is a constant depending on $\alpha, \beta$ and $\gamma$\,.
\end{lemma} 
\begin{theorem}\label{reg_1}
 Let $\ue$ be a solution of the penalized problem \eqref{Pe}. Then $\ue \in C^{1,\alpha}(\overline{B})$ for each $\alpha \in [0,1)$.
\end{theorem}
\begin{proof}
We choose $x_0 \in \overline{B}$ and fix $R<1$ such that $\abs{B_R(x_0)}<\omega_0$. Now let $0<r<R$. As mentioned above, we only consider the case $B_r(x_0)\cap\O(\ue)\neq\emptyset$. Now consider the comparison function $\hat{v}$ as in \eqref{eq:reg_1_11}. 
 Due to Lemma \ref{reg_1_la_1} and Lemma \ref{reg_1_la_3}, estimate \eqref{eq:reg_1_1} becomes
\[
 \int\limits_{B_r(x_0)\cap B}\abs{D^2\ue}^2dx \leq  C\,\left(\frac{r}{R}\right)^n\int\limits_{B_R(x_0)\cap B}\abs{D^2\ue}^2dx + C(n,\omega_0,\Le)R^{n-2+2\alpha}.
\]
 and applying Lemma~\ref{reg_1_la_2} leads to
\begin{align*}
 \int\limits_{B_r(x_0)\cap B}\abs{D^2\ue}^2dx &\leq C\,\left(\frac{r}{R}\right)^{n-2+2\alpha}\left(\Le+C(n,\omega_0,\Le)R^{n-2+2\alpha}\right). 
\end{align*}
Since $R$ was fixed, for every $i=1,\ldots,n$ there holds
\begin{equation}\label{eq:reg_1_12}
 \int\limits_{B_r(x_0)\cap B}\abs{\grad \partial_i\ue}^2dx \leq C(n,\omega_0,\Le)r^{n-2+2\alpha}  
\end{equation}
for every $B_r(x_0)$  with $r<R$ and $x_0 \in \overline{B}$.
Consequently, Theorem~\ref{morrey_theorem} implies $\partial_i \ue \in C^{0,\alpha}(\overline{B})$ and we finally obtain $\ue\in C^{1,\alpha}(\overline{B})$ for each $\alpha \in [0,1)$. 
\end{proof}
The $C^{1,\alpha}$ regularity of $\ue$ allows us to split $\partial\O$ in the two parts
\[
  \Geo := \{x \in \partial\mathcal{O}: \abs{\grad\ue(x)}=0\} \; \mbox{and} \; 
 \Ge :=  \{x \in \partial\mathcal{O} : \abs{\grad\ue(x)}>0\}.
\]
Obviously, $\Ge$ is part of a nodal line of $\ue$ since $\ue \in H^{2,2}_0(B)$.
Moreover, the continuity of $\grad \ue$ implies that $\Ge$ is a nullset with respect to the $n$-dimensional Lebesgue measure.
\begin{lemma}\label{lem:nodal_nullset}
 For every $\eps >0$ there holds $\mathcal{L}^n(\Ge)$=0.
\end{lemma}
\begin{proof}
Note that $\Ge$ is relativly open in $\partial\O$. Then, since $\ue \in C^{1,\alpha}$, the Implicit Function Theorem yields that $\Ge$ is (locally) a $C^{1,\alpha}$-graph. Therefore, $\mathcal{L}(\Ge)=0$. $\hfill \square$
\end{proof}
From now on, we set 
\begin{equation}\label{eq:ast}
   \Omega(\ue) := \O(\ue)\cup\Ge
\end{equation}
and call $\Geo ( = \partial\Omega(\ue))$ the free boundary.
\begin{remark}
 Note that $\Omega(\ue)$ is an open set in $\R^n$. Moreover, there holds
\begin{align*}
 \bilap\ue + \Le \lap\ue &=0 \mbox{ almost everywhere in } \Omega(\ue) \\ 
  \ue = \abs{\nabla \ue} &=0 \mbox{ in } \Ge = \partial\Omega(\ue).
\end{align*}
Furthermore, there holds $\abs{\Omega(\ue)} = \abs{\O(\ue)}$.
\end{remark}

\begin{remark}\label{est_Le}
  Note that $B$ contains a ball $B'$ with $\abs{B'}=\omega_0$. Consequently, for every $\eps>0$ there holds
\[
  \Le \leq \L(B') =: \L_{max}.
\]
Furthermore, the homothety property $t^2\L(tM)=\L(M)$ of the buckling eigenvalue implies
\[
  \L_{max} = \left(\frac{\omega_n}{\omega_0}\right)^\frac{2}{n}\L_1,
\]
where $\L_1$ is the first buckling eigenvalue of $B_1(0)\subset \R^n$ and $\omega_n := \mathcal{L}^n(B_1)$.
\end{remark}
Consequently, the $\alpha$-Hölder coefficients of $\ue$ and $\grad\ue$ are bounded independently of $\eps$, Moreover, we find that these bounds are also independent of $\alpha$. This is a consequence of Morrey's Dirichlet Growth Theorem. Hence, the Hölder coefficients of  $\ue$ and $\grad\ue$ for arbitrary $\alpha \in [0,1)$ are bounded by constants, which are independent of $\alpha$ and $\eps$, but depend on $n$ and $\omega_0$.
%%%% SECTION 3
\section{\texorpdfstring{$C^{1,1}$ Regularity of the Minimizers}{Lipschitz Regularity}}\label{ch:Reg}

In this section, we prove the $C^{1,1}$ regularity of the minimizers $\ue$ in $\overline{B}$. 
In particular, we show that  bound on the second order derivatives of a minimizer $\ue$ is independent of the parameter $\eps$.
For this purpose, we first show that $\lap\ue$ is bounded independently of $\eps$ almost everywhere in $B$. 
 In the sequel, we denote for $x \in B$
\[
\Ue (x) := \lap \ue(x)+\Le \ue(x)\,.
\]
Since $\ue$ solves \eqref{eq:var_eq}, $\Ue$ is a harmonic function almost everywhere in $\Omega(\ue)$.
\begin{theorem}\label{theorem:eli}
 Suppose $\ue \in H^{2,2}_0(B)$ is a minimizer of $\Je$. Assume $B_r(x_0) \subset \R^n$ with $x_0 \in \overline{B}$ and let 
\[
  \C_{B_r(x_0)} := \{v-\ue\in H^{2,2}_0(B_r(x_0)): \O(v)\cap B_
r(x_0)\subset\O(\ue)\cap B_r(x_0)\}\,.
\]
Then for each $v \in \C_{B_r(x_0)}$ there holds
\begin{equation}\notag
  \int\limits_{B_r(x_0)}\Ue\lap(\ue-v)\,dx \leq 0\,. 
\end{equation}
\end{theorem}
\begin{proof}
   Suppose $x_0 \in\overline{B}$ and $r>0$. 
  We consider the functional $\E : \C_{B_r(x_0)} \to\R$ given by
\[
 \E(v) := \int\limits_{B_r(x_0)\cap B}\abs{\lap v}^2dx - \Le \int\limits_{B_r(x_0)\cap B}\abs{\grad v}^2dx
\]
and minimize $\E$ in $\C_{B_r(x_0)}$. Without loss of generality, we assume that $\overline{B_r(x_0)}$ intersects $\partial\Omega(\ue)$. Otherwise the claim is obvious since $\Ue$ is harmonic in $\O(\ue)$.
 Let $v\in\C_{B_r(x_0)}$ be arbitrary. We define 
\[
 \hat{v} = \begin{cases}
  v, & \text{  in } B_r(x_0)\cap B\\
 \ue, & \text{  in } B\setminus B_r(x_0)
\end{cases}\,.
\] 
Note that $\hat{v}\in H^{2,2}_0(B)$, $ \O(\hat{v}) = [\O(\ue) \setminus (\O(\ue)\cap B_r(x_0)) ]\,\dot{\cup}\, \O(v)$ and $\abs{\O(\hat{v})} \leq \abs{\O(\ue)}$.
 Thus, we obtain
\[
  \Le \int\limits_B\abs{\grad\hat{v}}^2dx \leq \int\limits_B\abs{\lap \hat{v}}^2dx \; \aq \; \E(\ue) \leq \E(v)\,.
\]
Therefore, $\ue$ minimizes $\E$ in $\C_{B_r(x_0)}$.
Since $\C_{B_r(x_0)}$ is convex, for each $v\in\C_{B_r(x_0)}$ and every $t \in [0,1]$ there holds
\[
  \min_{v\in\C_{B_r(x_0)}}\E(v)= \E(\ue) \leq \E(t\ue+(1-t)v)\,.
\]
This implies
\begin{align*}
 \frac{d}{dt}\Bigg{|}_{t=1} \E(t\ue+(1-t)v) &\leq 0 
\aq\; \int\limits_{B_r(x_0)}\Ue \lap(\ue-v)\,dx \leq 0\,.
\end{align*}
\end{proof}
Now consider a ball $B_r(x_0)$, which intersects the free boundary. We will prove that the mean-value of $\abs{\lap \ue}^2$ over this ball is bounded independently of $x_0$, $r$ and $\eps$. This boundedness will be the essential observation in proving the $C^{1,1}$ regularity of a minimizer $\ue$. 
For the time being, we choose $x_0 \in B$ close to the free boundary. By 'close' we mean $\dist(x_0,\partial\Omega(\ue))<\frac{1}{4}$ in this context. Furthermore, we consider $r>0$ such that $\overline{B_r(x_0)}$ intersects the free boundary. Thus, there holds  $0<r<\frac{1}{4}$. Now we set
\begin{equation}\label{eq:a_r}
 \alpha_r := \frac{\ln(r)-\ln(1-r)}{\ln(r)}\,.
\end{equation}
Note that there holds  $\alpha_r \in \left(\frac{\ln(3)}{2\ln(2)},1\right)$ and $r^{1-\alpha_r} = 1-r$.
The following technical lemma is cited from \cite{GG1982}.
\begin{lemma}\label{iteration-lemma}
Let $f(t)$ be a nonnegative bounded function defined for $0\leq T_0\leq t \leq T_1$. Suppose that for $T_0 \leq t < s \leq T_1$ we have
\[
  f(t) \leq A(s-t)^{-l} + B + \theta f(s) ,
\]
where $A, B \mbox{ and }\theta$ are nonnegative constants and $\theta < 1$. Then there exists a constant $\gamma>1$, depending on $l$ and $\theta$ such that for every $\rho, R$, $T_0\leq\rho<R\leq T_1$ we have
\[
  f(\rho) \leq \gamma \left(A (R-\rho)^{-l}+B\right)\,.
\]  
\end{lemma}
\begin{theorem}\label{theorem:mw_bound}
Let $\ue$ be a solution of the problem \eqref{Pe}.
Suppose $x_0 \in\overline{B}$ and $0<r<\frac{1}{4}$ such that $\overline{B_r(x_0)}\cap \partial\Omega(\ue)\neq \emptyset$. Then there exists a constant $M_0>0$ such that there holds
 \[
  \fint\limits_{B_\frac{r}{2}(x_0)}\abs{\lap \ue}^2dx \leq M_0\,.
 \]
The constant $M_0$ depends on $n$ and $\omega_0$, but in particular, not on $x_0$, $r$ or $\eps$.  
\end{theorem}
\begin{proof}
 Let $x_0 \in \overline{B}$ and $r\in \left(0,\nicefrac{1}{4}\right)$ such that $\overline{B_r(x_0)}$ intersects  $\partial\Omega(\ue)$.  For the sake of convenience, we consider $x_0=0$. Note that there exists at least one point $\bar{x} \in \overline{B_r(x_0)}\cap\partial\Omega(\ue)$. Now let $\frac{r}{2}\leq t<s\leq r$.  
Consider the smooth  functions 
 $\mu \in C^\infty_0(B_s(0))$ and $\eta \in C^\infty(\R^n)$ with 
\[
  0\leq\mu\leq 1  \quad \text{and} \quad \mu \equiv 1 \text{ in } B_\frac{s+t}{2}(0)
\]
and 
\[
  0\leq\eta\leq 1\,,\; \eta \equiv 1 \text{ in }\R^n\setminus B_\frac{s+t}{2}(0) \quad \text{and} \quad \eta \equiv 0 \text{ in } B_t(0)\,.
\]
Remember that for cut-off functions like these there holds 
\[
 \abs{\partial^{\beta}\eta}\,,  \abs{\partial^{\beta}\mu}\leq \frac{C}{(s-t)^{\abs{\beta}}}
\]
for multi-indices $\beta$ with $0\leq\abs{\beta}\leq2$. The constant $C$ is independent of $s$ or $t$.
For $x \in B_s(0)$ we define $v(x) = \eta(x)(1-r \mu)\ue(x)$. Then $v \in \C_{B_s(0)}$. Therefore, Theorem~\ref{theorem:eli} implies
\[
 \int\limits_{B_s(0)}\Ue \lap(\ue-v)dx \leq 0\,.
\]
Consequently, there holds
\begin{align}
&\int\limits_{B_t}\Ue\lap \ue\,dx \leq -r\int\limits_{B_s\setminus B_\frac{s+t}{2}}\Ue\lap(\ue\,\mu)dx + \int\limits_{B_\frac{s+t}{2}\setminus B_t}\Ue\lap(\ue(\eta(1-r)-1))dx\,. \label{eq:lap_u_est_1}
\end{align}
We estimate the integrals on the right hand side separately. Since $\mu\geq 0$, we obtain for the first one
\begin{align*}
-r\int\limits_{B_s\setminus B_\frac{s+t}{2}}\Ue\lap(\ue\,\mu)dx 
&\leq  -r\int\limits_{B_s\setminus B_\frac{s+t}{2}}\lap \ue (\Le \ue\mu + 2\grad \ue.\grad\mu + \ue\lap\mu)\,dx \\ &\qquad+ \Le r\int\limits_{B_s\setminus B_\frac{s+t}{2}}2\abs{\ue}\abs{\grad \ue .\grad \mu} +\ue^2\abs{\lap\mu}\,dx\,.  
\end{align*} 
Applying Cauchy's inequality, we get
\begin{align*}
&-r\int\limits_{B_s\setminus B_\frac{s+t}{2}}\Ue\lap(\ue\mu)dx \leq  \frac{3}{2}\int\limits_{B_s\setminus B_\frac{s+t}{2}}\abs{\lap \ue}^2dx + 2\Le^2\int\limits_{B_s}\ue^2dx \\ & + 8r^2\int\limits_{B_s\setminus B_\frac{s+t}{2}}\abs{\grad \ue}^2\abs{\grad \mu}^2dx +
2r^2\int\limits_{B_s\setminus B_{\frac{s+t}{2}}}\ue^2\abs{\lap\mu}^2dx
 +C\frac{r^2}{(s-t)^2}\abs{B_s}\,,
\end{align*}
 This leads to
\begin{align*}
-r\int\limits_{B_s\setminus B_\frac{s+t}{2}}\Ue\lap(\ue\mu)dx\leq  \frac{3}{2}\int\limits_{B_s\setminus B_\frac{s+t}{2}}\abs{\lap \ue}^2dx +C  \abs{B_r}\left( 1 +\frac{r^4}{(s-t)^4}\right)\,.
\end{align*}
Next, we estimate the second integral on the right hand side of \eqref{eq:lap_u_est_1}. Since $\eta (1-r) -1\leq 0$, we achieve
\begin{align*}
&\int\limits_{B_\frac{s+t}{2}\setminus B_t}\Ue\lap(\ue(\eta(1-r)-1))dx \\ 
\leq &\int\limits_{B_\frac{s+t}{2}\setminus B_t}\lap \ue \left(2(1-r)\grad \ue.\grad\eta +\Le \ue (\eta(1-r)-1)+(1-r)\ue\lap\eta\right)dx \\ \qquad & + \Le(1-r)\int\limits_{B_\frac{s+t}{2}\setminus B_t}2\ue\grad \ue.\grad\eta + \ue^2\lap\eta\, dx\,.
\end{align*}
Applying Cauchy's inequality once more leads to
\begin{align*}
&\int\limits_{B_\frac{s+t}{2}\setminus B_t}\Ue\lap(\ue(\eta(1-r)-1))dx \\
\leq &\;\frac{3}{2}\int\limits_{B_\frac{s+t}{2}\setminus B_t}\abs{\lap \ue}^2 dx +2\Le^2\int\limits_{B_\frac{s+t}{2}\setminus B_t}\ue^2dx + 8(1-r)^2\int\limits_{B_\frac{s+t}{2}\setminus B_t}\abs{\grad \ue}^2\abs{\grad \eta}^2dx 
\\&+2(1-r)^2\int\limits_{B_\frac{s+t}{2}\setminus B_t}\ue^2\abs{\lap\eta}^2dx+ 2\Le(1-r)\int\limits_{B_\frac{s+t}{2}\setminus B_t}\ue\grad \ue.\grad\eta\, dx\,.
\end{align*}
Since $\grad \ue$ is $\alpha_r$-Hölder continuous, \eqref{eq:a_r} yields
\begin{align*}
\int\limits_{B_\frac{s+t}{2}\setminus B_t}\hspace{-2mm}\Ue\lap(\ue(\eta(1-r)-1))dx 
\leq  \frac{3}{2}\hspace{-2mm}\int\limits_{B_\frac{s+t}{2}\setminus B_t}\hspace{-2mm}\abs{\lap \ue}^2 dx +C\abs{B_r} \left(1+\frac{r^4}{(s-t)^4}\right).
\end{align*} Now we go back to \eqref{eq:lap_u_est_1} and obtain
\[
 \int\limits_{B_t}\abs{\lap \ue}^2 dx \leq  \frac{3}{2}\int\limits_{B_s\setminus B_t}\abs{\lap \ue}^2dx -\Le\int\limits_{B_t}\ue\lap \ue\,dx + C\abs{B_r} \left(1+\frac{r^4}{(s-t)^4}\right),
\]
Again, we apply Cauchy's inequality and achieve
\[
\int\limits_{B_t}\abs{\lap \ue}^2\,dx \leq 3\int\limits_{B_s\setminus B_t}\abs{\lap \ue}^2 dx + C\abs{B_r}\left(1+\frac{r^4}{(s-t)^4}\right).
\]
Now we add three times the left hand side of the above inequality to both sides of the inequality. This 'fills the hole' in the domain over which the integral on the right hand side is taken. We obtain
\[
\int\limits_{B_t}\abs{\lap \ue}^2\,dx \leq \frac{3}{4}\int\limits_{B_s}\abs{\lap \ue}^2 dx + C\abs{B_r}\left(1+\frac{r^4}{(s-t)^4}\right).
\]
Thus, we may apply Lemma \ref{iteration-lemma} and  obtain
\[
\int\limits_{B_\frac{r}{2}}\abs{\lap \ue}^2dx \leq C(n,\omega_0)\abs{B_r}\,.
\]
Note that the constant $C$ is independent of $x_0$, $r$ and, in particular, of $\eps$. It only depends on $n$ and $\omega_0$.
\end{proof}
Theorem \ref{theorem:mw_bound} allows  to show that $\lap \ue$ is  bounded almost everywhere in $B$. 
To prove this, we set
\begin{equation}\label{eq:Omega_ast}
\Omega^\ast(\ue):= \left\{ x \in \Omega(\ue): \dist(x,\partial\Omega(\ue))<\frac{1}{8}\right\}.  
\end{equation}
and divide the closure of $\Omega(\ue)$ in three parts (provided that $\Omega^\ast(\ue) \neq \Omega(\ue)$): 
\begin{enumerate}
 \item the inner part $\Omega(\ue)\setminus \Omega^\ast(\ue)$, in which the distance of each point to the free boundary is 'large',
 \item the inner neighbourhood $\Omega^\ast(\ue)$ of the free boundary, which contains points with sufficiently small distance to the free boundary,
 \item the free boundary itself.
\end{enumerate}
In each of this sets we establish a bound on $\lap\ue$. 
If $\Omega^\ast(\ue) = \Omega(\ue)$, it suffices to consider $\Omega(\ue)$ and the free boundary. In this case, a separate analysis of an inner part is not necessary (see Remark~\ref{remark:3parts}).
Establishing a separate bound on the free boundary is necessary since we still do not know whether the free boundary is a nullset with respect to the $n$-dimensional Lebesgue measure or not.
Due to well-known results in the theory of Partial Differential Equations, we get an inner bound for the second order derivatives of $\ue$, which is independent of the parameter $\eps$.
 \begin{lemma}\label{la:inner_bound}
  There exists a constant $M_1=M_1(n,\omega_0)>0$ such that for each minimizer $\ue$ there holds 
  \[
   \abs{D^2\ue(x)}\leq M_1 \; \text{ for each }x \in \Omega(\ue)\setminus \Omega^\ast(\ue).
  \]
 \end{lemma}
Our next step is establishing a bound on $\lap \ue$ in $\Omega^\ast(\ue)$. Due to Theorem~\ref{theorem:mw_bound}, this bound is independent of  $\eps$, too.
\begin{lemma}\label{tub_bound}
 There exists a constant $M_2=M_2(n,\omega_0)>0$ such that for each minimizer $\ue$ there holds 
 \[
  \abs{\lap \ue(x)} \leq M_2  \text{ for almost every }x \in \Omega^\ast(\ue)\,.
 \]
\end{lemma}
\begin{proof}
  Recall that $\Ue$ is harmonic almost everywhere in $\Omega(\ue)$. Choose $x_0 \in \Omega^\ast(\ue)$ and $\delta>0$ such that
 \[
  \Ue(x_0) = \fint\limits_{B_d(x_0)}\Ue(x)\,dx\,,
 \]
where $d := \dist(x_0,\partial\Omega(\ue))-\delta$. Note that $\dist(x_0,\partial\Omega(\ue))< d+ d^{n+1}<\frac{1}{4}$ for $\delta$ sufficiently small since $x_0 \in \Omega^\ast(\ue)$.
We extend the ball $B_d(x_0)$ to $B_{d+d^{n+1}}(x_0)$ and correct our error immediately. Hence,
\[
 \Ue(x_0) = \frac{1}{\abs{B_d}}\int\limits_{B_{d+d^{n+1}}(x_0)}\Ue(x)\,dx - \frac{1}{\abs{B_d}}\int\limits_{B_{d+d^{n+1}}(x_0)\setminus B_d(x_0)}\Ue(x)\,dx\,.
\]
Now we take absolute values and estimate applying Hölder's inequality. Since $d$ is bounded from above by $\frac{1}{4}$, we obtain
\[
 \abs{\Ue(x_0)} \leq C(n) \sqrt{\fint\limits_{B_{d+d^{n+1}}(x_0)}\abs{\lap \ue}^2dx + C(n,\omega_0)} + C(n,\omega_0)\,.
\]
 Thus, applying Theorem \ref{theorem:mw_bound} yields
 \begin{equation}\label{eq:tub_bound}
  \abs{\Ue(x_0)} \leq C(n,\omega_0,M_0)\leq C(n,\omega_0),
 \end{equation}
where $M_0$ is the constant introduced in Theorem~\ref{theorem:mw_bound}. Since $x_0 \in \Omega^\ast(\ue)$ was chosen arbitrarily, and the constant in \eqref{eq:tub_bound} is independent of $x_0$ and $d_0$, the claim is proven.
\end{proof}
\begin{remark}\label{remark:3parts}
 If $\Omega^\ast(\ue) = \Omega(\ue)$, the previous lemma gives an estimate for each $x \in \Omega(\ue)$. In this case, we could omit Lemma~\ref{la:inner_bound}. In the proof of Lemma~\ref{tub_bound}, we need to control the distance to the free boundary by a fixed quantity. This is why we need a separate analysis of $\Omega(\ue)\setminus\Omega^\ast(\ue)$ if $\Omega^\ast(\ue)\neq\Omega(\ue)$.
\end{remark}
Finally, yet importantly, we need a uniform bound for $\lap \ue$ on the free boundary $\partial\Omega(\ue)$. If the free boundary was a Lebesgue nullset, we obviously would not need this consideration. 
\begin{lemma}\label{b_bound}
 Suppose $\mathcal{L}^n(\partial\Omega(\ue))>0$. Then for almost every $x \in \partial \Omega(\ue)$ there holds
 \[
  \abs{\lap \ue(x)} \leq \sqrt{M_0}\,,
 \]
where $M_0$ is the constant introduced in Theorem~\ref{theorem:mw_bound}.
\end{lemma}
\begin{proof}
 Since $\lap \ue \in L^2(B)$ and we suppose $\mathcal{L}^n(\partial\Omega(\ue))>0$, almost every $x \in \partial \Omega(\ue)$ is a Lebesgue Point of $\abs{\lap \ue}^2$. Hence, Theorem \ref{theorem:mw_bound} implies  
 \[
  \abs{\lap \ue(x)}^2 = \lim_{r\to0}\fint\limits_{B_r(x)}\abs{\lap \ue(y)}^2dy \leq M_0 \quad \mbox{for almost every $x \in \partial\Omega(\ue)$}\,.
 \]
\end{proof}
Joining the Lemmata \ref{la:inner_bound}, \ref{tub_bound} and \ref{b_bound} we achieve the main theorem of this section.
\begin{theorem}\label{lap_u_bounded}
 There exists a constant $M>0$, depending on $n$ and $\omega_0$, such that for each minimizer $\ue$ of  the functional $\Je$ there holds $\norm{\lap \ue}_{L^\infty(B)} \leq M$. 
\end{theorem}
To prove the $C^{1,1}$ regularity we show that the second order derivatives of $\ue$ are bounded almost everywhere in $B$. 
We again divide $\overline{\Omega(\ue)}$ in the three parts mentioned in the beginning of the previous  section (provided that $\Omega^\ast(\ue)\neq\Omega(\ue)$). 
Due to Lemma~\ref{la:inner_bound}, we already know that the second order derivatives of $\ue$ are bounded in $\Omega(\ue)\setminus\Omega^\ast(\ue)$.
Using an idea of J.~Frehse in \cite{frehse_73}, we find a uniform bound in the set $\Omega^\ast(\ue)$. In a similar way, L.\,A.~Caffarelli and A.~Friedman argue in \cite{CF_79}. 
Finally, we establish a bound on the second order derivatives on the free boundary. 
Note that the essential device for proving the next lemma is Theorem~\ref{lap_u_bounded}.
\begin{lemma}\label{tub_bound_sod}
  There exists a constant $M_3 = M_3(n,\omega_0)>0$ such that for each minimizer $\ue$ there holds
\[
  \abs{D^2\ue(x)} \leq M_3 \text{ for every } x \in \Omega^\ast(\ue)\,.
\]
\end{lemma}
\begin{proof}
 Let $G_n: \R^n \to \R$ be the biharmonic fundamental solution, i.e.
\begin{equation}\notag
  G_n(x) := \begin{cases}
 \abs{x}^2\left(\ln(x)-1\right), & n=2  \\
 -\abs{x}, & n=3
\end{cases}\,. 
\end{equation}
Now we choose $x_0 \in \Omega^\ast(\ue)$ and set $r:=  \frac{1}{2}\dist(x_0,\partial\Omega(\ue))$.
We consider the cut-off functions $\eta \in C^\infty_0(B_r(x_0))$ and $\mu \in C^\infty(B_r(x_0))$ satisfying 
\[\begin{split}
  &0\leq \eta, \mu \leq 1\,,\quad \eta \equiv 1 \text{ in } B_\frac{r}{2}(x_0), \\ 
& \mu \equiv 1 \text{ in } B_r\setminus B_\frac{r}{2}(x_0),\quad \text{ and } \mu \equiv 0 \text{ in } B_\frac{r}{4}(x_0)\,.
\end{split}
\]
Setting $\zeta := \eta(1-r\mu)$, for each $x \in B_\frac{r}{8}(x_0)$ there holds
\begin{equation}\label{eq:sod_1}
  \ue(x) = \int\limits_{B_r(x_0)}G_n(x-y)\bilap(\ue\zeta)(y)\,dy \,. 
\end{equation}
For $1\leq k,l\leq n$ we define 
\[
  L_{kl} := \partial_{x_k}\partial_{x_l} - \frac{1}{2}\delta_{kl}\lap_x\,.
\] 
Straightforward computation shows that 
\begin{equation}\label{eq:sod_2}
  L_{kl}G_n(x) = c(n)\frac{x_kx_l}{\abs{x}^n} \quad \text{and} \quad \abs{\grad L_{kl}G_n(x)}\leq \frac{c(n)}{\abs{x}^{n-1}}\,.
\end{equation}
We apply the operator $L_{kl}$ on both sides of \eqref{eq:sod_1}. On the right hand side, this yields
\begin{equation}\label{eq:sod_3}
\begin{split}
  \int\limits_{B_r(x_0)}L_{kl}G_n(x-y)\bilap(\ue\zeta)(y)\,dy = \int\limits_{B_\frac{r}{4}(x_0)}L_{kl}G_n(x-y)\bilap \ue(y)dy \\
 +\int\limits_{B_\frac{r}{2}\setminus B_\frac{r}{4}(x_0)}L_{kl}G_n(x-y)\bilap(\ue(y)(1-r\mu(y)))\,dy \\
+(1-r)
\int\limits_{B_r\setminus B_\frac{r}{2}(x_0)}L_{kl}G_n(x-y)\bilap(\ue\eta)(y)\,dy  \,.
\end{split}\end{equation}
Following the lines of Frehse \cite{frehse_73}, we estimate the integrals on the right hand side seperately. Thereby, we make use of the partial differential equation the minimizer $\ue$ satisfies almost everywhere and apply Theorem \ref{lap_u_bounded}. After some straight forward computations we obtain 
\[
 \abs{\int\limits_{B_r(x_0)}L_{kl}G_n(x-y)\bilap(\ue\zeta)(y)dy} \leq C(n,\omega_0).
\]
Let us emphasize once more that the constant $C$ is independent of $l,\,k,\,x,\,x_0$, and $r$. Hence, applying $L_{kl}$ on both sides of \eqref{eq:sod_1} leads to
\[
 \abs{L_{kl} \ue(x)} \leq C(n,\omega_0).
\]
This proves the claim.
\end{proof}
It remains to  establish a bound on the second order derivatives of $\ue$ on the free boundary, provided that $\mathcal{L}^n(\partial\Omega(\ue))>0$. 
For this purpose, we have to show that the mean-value of $\abs{D^2\ue}$  over balls with their centre in  the free boundary is uniformly bounded. This is done using the same construction as in the proof of Theorem \ref{theorem:mw_bound}. Then a Lebesgue Point argument as in Lemma \ref{b_bound} leads to a uniform bound on $\abs{D^2\ue}$ on $\partial\Omega(\ue)$. Thus, we leave the proof to the reader.
\begin{lemma}\label{b_bound_sod1}
 Let $\ue$ be a minimizer of $\Je$. Suppose $x_0\in\partial\Omega(\ue)$ and $r>0$ such that $B_r(x_0)\subset \Omega^{\ast}(\ue)$. Then there exists a constant $M_4 = M_4(n,\omega_0)$ such that 
 \[
  \fint\limits_{B_r(x_0)}\abs{D^2\ue}^2dx \leq M_4.
 \]
\end{lemma}
The boundedness of $\abs{D^2\ue}$ on the free boundary now follows analogously to Lemma~\ref{b_bound}. 
\begin{corollary}\label{b_bound_sod}
Let $\ue$ be a minimizer of the functional $\Je$ and  suppose \\$\mathcal{L}^n(\partial\Omega(\ue))>0$. Then  for almost every $x \in\partial\Omega(\ue)$ there holds $\abs{D^2\ue(x)}\leq \sqrt{M_4}.$
\end{corollary}
The main theorem of this section now follows from Lemma~\ref{la:inner_bound}, Lemma~\ref{tub_bound_sod}, and Corollary~\ref{b_bound_sod}\,.
\begin{theorem}
 Suppose $\ue \in H^{2,2}_0(B)$ is a minimizer of the functional $\Je$. Then $\ue~\in~C^{1,1}(\overline{B})$. 
\end{theorem}
\begin{remark}\label{remark:bound_indep_eps}
Let us emphasize once more that the second order derivatives of $\ue$ are bounded by a constant which is independent of the parameter $\eps$. This is the fundamental observation for proving the equivalence of the penalized and the original problem.
\end{remark}
%%%%
%%%% SECTION 4
\section{Equivalence of the Penalized and the Original Problem}\label{ch:prob_equiv}

In this section, we achieve a critical parameter $\eps_0>0$ such that $\ue$ solves the original problem \eqref{P} if we choose $\eps<\eps_0$. In this way, the problems \eqref{Pe} and \eqref{P} can be treated as equivalent. The uniform bound on the second order derivatives of $\ue$, which we established in the previous section, is crucial for proving the equivalence of the problems \eqref{Pe} and \eqref{P}. 
In the end of this section, we compute the first variation of the functional $\J$.

Following the lines  of T.~Stepanov and P.~Tilli in \cite{SteTil02}, we establish an Euler-type equation for $\ue$. This Euler-type equation helps to quantify the critical value $\eps_0$. 
\begin{lemma}\label{ETE}
  Let $\phi \in H^{1,2}_0(\R^n)\cap H^{2,2}(\R^n)$ and $g \in C^{1,1}(\R)$ with $g(0)=0$. Assume $\ue \in H^{2,2}_0(B)$ is a minimizer of $\Je$. Then there holds
\[
 \int\limits_{B} \lap \ue\,\lap(\phi\,g(\ue)) - \Le \grad \ue.\grad(\phi\,g(\ue))\,dx = 0.
\]
\end{lemma}
\begin{proof}
We set $\Psi_\delta(x) := \ue(x) + \delta \, \phi(x)g(\ue(x))$, where $\delta$ is arbitrary. Note that $\Psi_\delta \in H^{2,2}_0(B)$ and there holds 
$\O(\Psi_\delta)\subset\O(\ue) $. This yields 
$f_\epsilon(|\O(\Psi_\delta)|) \leq f_\epsilon(\abs{\O(\ue)})$
and comparing the $\Je$-energies of $\ue$ and $\Psi_\delta$ proves the claim.
\end{proof}
\begin{theorem}\label{theorem:vol_satis}
 Suppose $\ue\in H^{2,2}_0(B)$ is a minimizer of $\Je$. Then there exists an $\epsilon_0>0$ such that for $\epsilon<\epsilon_0$ there holds $\abs{\O(\ue)} =  \omega_0$\,.
\end{theorem}
\begin{proof}
 From Theorem~\ref{supp_bigger} we know that $\abs{\O(\ue)}\geq\omega_0$ holds for each $\eps>0$. Hence, we need to disprove $\abs{\O(\ue)}>\omega_0$ for $\eps$ sufficiently small. Therefore, let us assume that there is a minimizer $\ue$ with $\abs{\O(\ue)}>\omega_0$. 
Consequently, we can choose an $x_\eps \in \O(\ue)$ such that $r_\eps:= \dist(x_\eps,\partial\O(\ue)) = \dist(x_\eps,\partial\Omega(\ue))$ and 
\[
   \abs{\O(\ue)}-\abs{B_{r_\eps}(x_\eps)} >\omega_0.
\]
We construct a comparison function $v_\eps \in H^{2,2}_0(B)$, which is equal to $\ue$ outside of $B_{2r_\eps}(x_\eps)$ and vanishes in $B_{r_\eps}(x_\eps)$. For this reason,  
 let $\eta \in C^\infty_c(B_{2r_\eps}(x_\eps))$ be a cut-off function with $0\leq \eta\leq 1$ and $\eta \equiv 1 $ in $B_{r_\eps}(x_ \eps)$. Defining $v_\eps := \ue-\eta\ue $, we obtain the desired comparison function. Moreover, we find
\[
 \abs{\O(v_\eps)}=\abs{\O(\ue)}-\abs{B_{r_\eps}(x_\eps)}>\omega_0.
\]
Subsequently, the monotonicity of the penalization term $\fe$ implies
\[
 \fe(\abs{\O(v_\eps)}) - \fe(\abs{\O(\ue)}) = -\frac{1}{\eps}\abs{B_{r_\eps}(x_\eps)}.
\]
Now comparing the $\Je$-energies of $\ue$ and $v_\eps$ and applying Lemma~\ref{ETE} leads to
\begin{equation}\label{eq:prob_0}
 \frac{\abs{B_{r_\eps}(x_\eps)}}{\eps}\int\limits_{B}\abs{\grad v_\eps}^2dx \leq \Le \int\limits_{B_{r_\eps(x_\eps)}}\abs{\grad \ue}^2dx + \int\limits_{B_{2r_\eps}\setminus B_{r_\eps}(x_\eps)}\abs{\lap(\eta\ue)}^2dx.
\end{equation}
In the following, we will show that there exists a constant $C(n,\omega_0)$ such that 
\begin{equation}\label{eq:prob_1}
 \int\limits_{B_{2r_\eps}\setminus B_{r_\eps}(x_\eps)}\abs{\lap(\eta\ue)}^2dx \leq C(n,\omega_0)\abs{B_{2r_\eps}\setminus B_{r_\eps}(x_\eps)}.
\end{equation}
In addition, we prove that 
\begin{equation}\label{eq:prob_2}
\int\limits_B\abs{\grad v_\eps}^2dx \geq \frac{1}{2} \quad  \text{for $r_\eps$ sufficiently small.}
\end{equation}
Combining these two estimates with \eqref{eq:prob_0} yields
\begin{align*}
   \frac{1}{2\eps}\abs{B_{r_{\eps}}(x_\eps)} \leq \Le \int\limits_{B_{r_\eps(x_\eps)}}\abs{\grad \ue}^2dx + C(n,\omega_0) \abs{B_{2r_\eps}\setminus B_{r_\eps}(x_\eps)}.
\end{align*}
Due to the Lipschitz continuity of $\grad \ue$ and Remark~\ref{est_Le} we obtain
\begin{align}\label{eq:equiv}
  \frac{1}{2\eps}\abs{B_{r_{\eps}}(x_\eps)} \leq C(n,\omega_0) \abs{B_{2r_\eps}(x_\eps)}.
\end{align}
Hence, if $\abs{\Omega(\ue)}>\omega_0$, then
\[
  \eps_0:= \frac{1}{C(n,\omega_0)} \leq \eps\,.
\]
In other words, if we choose $\eps<\eps_0$, then \eqref{eq:equiv} cannot hold true and we obtain $\O(\ue)=\omega_0$.
Thus, to finish the proof of this theorem, it remains to establish the estimates \eqref{eq:prob_1} and \eqref{eq:prob_2}.
We start by proving the estimate \eqref{eq:prob_1}. Applying Cauchy's inequality we obtain
\[
  \int\limits_{B_{2r_\eps}\setminus B_{r_\eps}(x_\eps)}\abs{\lap(\eta\ue)}^2dx \leq C  \int\limits_{B_{2r_\eps}\setminus B_{r_\eps}(x_\eps)} \abs{\lap \ue}^2\eta^2 + \abs{\grad\ue}^2\abs{\grad\eta}^2+\ue^2\abs{\lap\eta}^2dx.
\]
Now note that Lemma~\ref{b_bound_sod1}, together with Theorem~\ref{morrey_theorem}, implies
\[
 \sup_{B_{2r_\eps}(x_\eps)}\abs{\grad \ue}\leq C(n,\omega_0,M_4)\,r_\eps
\quad\mbox{and } 
 \sup_{B_{2r_\eps}(x_\eps)}\abs{ \ue}\leq C(n,\omega_0,M_4)\,r_\eps^2.
\]
Furthermore, $ \lap \ue$ is bounded independently of $\eps$.
Subsequently, we obtain
\begin{align*}
 \int\limits_{B_{2r_\eps}\setminus B_{r_\eps}(x_\eps)}\abs{\lap(\eta\ue)}^2dx &\leq C(n,\omega_0)\abs{B_{2r_\eps}(x_\eps)\setminus B_{r_\eps}(x_\eps)}.
\end{align*}
This proves \eqref{eq:prob_1}.
In the same way as above, we estimate 
\[
  \int\limits_B\abs{\grad v_\eps}^2dx \geq 1- C(n,\omega_0) r_\eps^{n+2}.
\]
Thus,  for $r_\eps$ sufficiently small we obtain \eqref{eq:prob_2}.
\end{proof}
\begin{remark}\label{remark:why}
 Note that it is essential to choose an $x_\eps \in \O(\ue)$ instead of $x_\eps \in \partial\O(\ue)$ in the previous proof. If we assumed $x_\eps\in\partial\O(\ue)$, we would not obtain the estimate~(\ref{eq:equiv}), but
 \[
  \frac{1}{2\eps}\abs{B_{r_\eps}(x_\eps)\cap\O(\ue)} \leq C(n,\omega_0)\abs{B_{2r_\eps}(x_\eps)\cap\O(\ue)}.
 \]
However, we do not know if this inequality is contradictory for small $\eps$. If $x_\eps$ is located in a very thin part of $\O(\ue)$, e.\,g. a thin cusp, the above estimate might even be true for every $\eps$. 
In the previous proof, we could avoid this difficulty by choosing an $x_\eps \in \O(\ue)$. Thus, on the left hand side of inequality~(\ref{eq:equiv}) the full measure of the ball $B_{r_\eps}$ occurs. Consequently, the radius $r_\eps$ and the centre $x_\eps$ cancel and the only dependence on $\eps$ is contained in the factor $\nicefrac{1}{2\eps}$. In this way, we obtain the desired contradiction.
\end{remark}
In the sequel, we always consider $\eps<\eps_0$. Consequently, $\O(\ue)$ fulfils the volume condition, i.\,e. $\abs{\O(\ue)}=\abs{\Omega(\ue)}=\omega_0$ and $\ue$ is a solution of the original problem $\eqref{P}$. In this way, we can treat the penalized problem \eqref{Pe} and the original problem \eqref{P} as equivalent. 
For this reason, we omit the index $\eps$ and write $u$ instead of $\ue$ and $\L$ instead of $\Le$. 
\begin{remark}\label{remark:min}
 Note that $u$ minimizes the functional $\J$ not only among all $v\in H^{2,2}_0(B)$ with $\abs{\O(v)}~=~\omega_0$, but among all $v~\in~H^{2,2}_0(B)$ with $\abs{\O(v)}~\leq~ \omega_0$, as well. 
\end{remark}

\begin{remark}
Since
\[
\Le = \L(\Omega(\ue)) = \min\{\J(v): v \in H^{2,2}_0(\Omega(\ue))\},
\]
$\Le$ is the buckling eigenvalue of the domain $\Omega(\ue)$.
\end{remark}

At this point, we can make a first statement regarding the shape of $\Omega(u)$. 
\vspace{.3cm}
\begin{corollary}\label{corollary:connected}
 The optimal domain $\Omega(u)$ 
is connected. 
\end{corollary}
\begin{proof}
 Let us assume that $\Omega(u)$ is the union of two disjoint sets $\Omega_1$
and $\Omega_2$. Let $u$ be the corresponding eigenfunction on $\Omega(u)$. Then according to the previous results there holds $u \in  H^{2,2}_0(\Omega_1)$ and $u\in H^{2,2}_0(\Omega_2)$. Furthermore, $\abs{\Omega_1},\,\abs{\Omega_2}< \omega_0$. Now we define 
\[
 u_1 := \begin{cases}
             u,& \text{in } \Omega_1 \\
             0, &\text{in } B\setminus \Omega_1 
            \end{cases} 
\quad \text {and} \quad 
u_2 := \begin{cases}
             u,& \text{in } \Omega_2 \\
             0, &\text{in } B\setminus \Omega_2 
            \end{cases} \,.
\]
Then $u_1,\,u_2 \in H^{2,2}_0(B)$. Since $\abs{\Omega(u_1)}<\omega_0$, Remark \ref{remark:min} and Theorem \ref{theorem:vol_satis} imply
\[
\L \int\limits_B\abs{\grad u_1}^2dx < \int\limits_{B}\abs{\lap u_1}^2 dx.
\]
Due to the normalization $\norm{\grad u}_{L^2(B)}=1$ this is equivalent to
\[
\int\limits_B\abs{\lap u_2}^2dx -\L \int\limits_B\abs{\grad u_2}^2dx<0\,. 
\]
This means $\J(u_2)<\L$, which is contradictory to the minimality of $\L$.
\end{proof}
Next, we compute the first variation of the functional $\Je$ in 
\[
  \Omega^+(u) := \interior\{x\in B: u(x)\geq 0\} \quad \text{and} \quad \Omega^-(u) := \interior\{x\in B: u(x)\leq 0\}\,.
\]
Crucial for this computation is the $C^{1,1}$ regularity of a minimizer $u$ and the embedding theorem \ref{theorem:embed}. For the proof of this theorem we refer to \cite{hitch} or \cite{Kn13}, Chapter A.1. 
\begin{theorem}\label{theorem:embed}
 Let $v\geq 0$ be in $C^{1,1}(B)\cap H^{2,2}_0(B)$. Furthermore, let $p\geq 4$, $B_R(x_0) \subset B$ for some $x_0 \in B$ and $R>0$. Then $\sqrt{v} \in H^{1,p}(B_R(x_0))$ and the following estimate holds
  \[
    \int\limits_{B_R(x_0)} \abs{\grad{\sqrt{v}}}^p\,dx \; \leq \; C\,\norm{v}^\frac{p}{2}_{C^{1,1}(B)}\,,
  \]
  where $C$ does not depend on $v$.
\end{theorem}
\begin{theorem}\label{ELI_nonneg}
 Suppose $B_r(x_0) \subset\Omega^+(u)$. Then for each $\phi \in C^\infty_{c}(B_r(x_0))$ with $\phi\geq 0$ there holds
 \[
  \int\limits_{B_r(x_0)\cap\{u>0\}}\U\lap\phi\,dx \leq 0\,.
 \]
\end{theorem}
\begin{proof}
 Since $\U$ is harmonic almost everywhere in $\Omega(u)$, it is sufficient to assume $B_r(x_0)\cap \partial\Omega(u)\neq \emptyset$. The main difficulty in this proof is the choice of an appropriate test function. We require a perturbation $v$ of $u$ with $v\in H^{2,2}_0(B)$, which fulfils $\abs{\O(v)}\leq\abs{\O(u)}$. A perturbation, which enlarges the support, inhibits a reasonable comparison of the $\J$-energies because of the monotonicity of the penalization term. 
Let us consider $\phi\in C^\infty_c(B_r(x_0))$ with $\phi\geq 0$. The natural way to choose a test function, which decreases the Lebesgue measure of $\O(u)$, would be $(u- \delta\phi)_+$ for some small positive $\delta$. However, in general this function is not in $H^{2,2}(B_r(x_0))$. Hence, we need a regularization of the positive part. 
Therefore, we define for $\delta > 0$ 
\[
  h_\delta(s) := \begin{cases}
  0, & s\leq 0 \\
  -\frac{\delta}{1-\delta}\frac{s^2}{u}+\frac{s^{1+\delta}}{u^\delta(1-\delta)}, &0<s<u \\
 s, & s\geq u
\end{cases}.
\]
Now consider $\phi\in C^\infty_c(B_r(x_0))$ with $\phi\geq 0$. Then $h_\delta(u-\delta\phi) \in H^{2,2}_0(B)$ and $\O(h_\delta(u-\delta\phi)) \subset \O(u)$. 
Note that in $\{0< u-\delta\phi < u\}$ there holds  
\begin{equation}\label{eq:h'}
  h'_\delta(u-\delta\phi) = 1 + O(\delta).
\end{equation}
Thus, 
\begin{equation}\label{eq:grad_h}
 \abs{\grad h_\delta(u-\delta\phi)}^2 =  \abs{\grad (u-\delta\phi)}^2 + O(\delta),
\end{equation}
where $O(\delta)$ collects all terms, which vanish in the limit as $\delta$ tends to zero. Furthermore, in $\{0< u-\delta\phi < u\}$ we find that 
\begin{equation}\label{eq:h''}
 (u-\delta\phi)^2(h''_\delta(u-\delta\phi))^2 
 \leq 50 \delta^2
\end{equation}
if we choose $0 <\delta < \frac{1}{2}$.
Now we compare the $\J$-energy of $h_\delta(u-\delta\phi)$ with $\L$. The minimality of $u$ implies
\[
\L\int\limits_B\abs{\grad h_\delta(u-\delta\phi)}^2dx \leq \int\limits_B \abs{\lap h_\delta(u-\delta\phi)}^2 dx.
\]
Consequently, we get the local estimate
\begin{equation}\label{eq:local}
 \L \int\limits_{B_r(x_0)} \abs{\grad h_\delta(u-\delta\phi)}^2 - \abs{\grad u}^2dx \leq \int\limits_{B_r(x_0)}\abs{\lap h_\delta(u-\delta\phi)}^2 - \abs{\lap u}^2dx.
\end{equation}
For the sake of convenience, we write $B_r$ instead of $B_r(x_0)$ and $h_\delta$ instead of $h_\delta(u-\delta\phi)$. The first integral on the left hand side of the above inequality can be reformulated as
\begin{align*}
\int\limits_{B_r}\abs{\grad h_\delta}^2dx 
=\hspace{-4mm}\int\limits_{B_r\cap\{0<u-\delta\phi<u\}}\hspace{-4mm}\left(h'^2_\delta-1\right)\abs{\grad(u-\delta\phi)}^2dx +\hspace{-4mm} \int\limits_{B_r\cap\{0<u-\delta\phi\}}\hspace{-4mm}\abs{\grad(u-\delta\phi)}^2dx 
\end{align*}
Since $\abs{\{0< u-\delta\phi < u\}}$ tends to zero as $\delta$ tends to zero, identity \eqref{eq:grad_h} implies
\begin{equation}\label{eq:left}
\int\limits_{B_r}\abs{\grad h_\delta}^2dx = \int\limits_{B_r\cap\{0<u-\delta\phi\}}\hspace{-4mm}\abs{\grad(u-\delta\phi)}^2dx  +o(\delta),
\end{equation}
where $o(\delta)$ collects all terms such that $\frac{o(\delta)}{\delta}\to 0$ as $\delta \to 0$.  It remains to study
\begin{align*}
 \int\limits_{B_r}\abs{\lap h_\delta}^2dx
= \hspace{-4mm}\int\limits_{B_r\cap\{0<u-\delta\phi<u\}}\hspace{-4mm}\abs{\lap h_\delta}^2-\abs{\lap (u-\delta\phi)}^2dx +\hspace{-7mm} \int\limits_{B_r\cap\{0<u-\delta\phi\}}\hspace{-4mm}\abs{\lap (u-\delta\phi)}^2dx
\end{align*}
This will be more challenging than study the integral before. Note that there holds
\begin{align*}
 &\int\limits_{B_r\cap\{0<u-\delta\phi<u\}}\hspace{-4mm}\abs{\lap h_\delta}^2-\abs{\lap (u-\delta\phi)}^2dx \\
\leq\quad & \int\limits_{B_r\cap\{0<u-\delta\phi<u\}}\hspace{-4mm}h''^2_\delta\abs{\grad(u-\delta\phi)}^4 + \left(h'^2_\delta-1\right)\abs{\lap(u-\delta\phi)}^2dx \\&+ C\sqrt{\abs{\{0<u-\delta\phi<u\}}}\sqrt{\int\limits_{B_r\cap\{0<u-\delta\phi<u\}}\hspace{-4mm} h''^2_\delta\abs{\grad(u-\delta\phi)}^4dx},
\end{align*}
where we used the $C^{1,1}$ regularity of $u$.  In order to apply estimate \eqref{eq:h''}, we use the identity
\[
   \abs{\grad(u-\delta\phi)}^4 = 16(u-\delta\phi)^2\abs{\grad\sqrt{u-\delta\phi}}^4.
\] 
Subsequently, applying \eqref{eq:h'} and \eqref{eq:h''}, we can proceeed to
\begin{align*}
 &\int\limits_{B_r\cap\{0<u-\delta\phi<u\}}\hspace{-4mm}\abs{\lap h_\delta}^2-\abs{\lap (u-\delta\phi)}^2dx 
\leq C\delta^2\hspace{-6mm}\int\limits_{B_r\cap\{0<u-\delta\phi<u\}}\hspace{-4mm}\abs{\grad\sqrt{u-\delta\phi}}^4 dx\\ \;+ & C\delta \sqrt{\abs{\{0<u-\delta\phi<u\}}}\sqrt{\int\limits_{B_r\cap\{0<u-\delta\phi<u\}}\hspace{-4mm} \abs{\grad\sqrt{u-\delta\phi}}^4dx} + o(\delta).
\end{align*}
Now Theorem \ref{theorem:embed} implies 
\[
\int\limits_{B_r\cap\{0<u-\delta\phi<u\}}\abs{\lap h_\delta}^2-\abs{\lap (u-\delta\phi)}^2dx = o(\delta)
\]
and, subsequently, 
\begin{equation}\label{eq:right}
\int\limits_{B_r}\abs{\lap h_\delta}^2dx = \int\limits_{B_r\cap\{0<u-\delta\phi\}}\abs{\lap(u-\delta\phi)}^2dx +o(\delta).
\end{equation}
Going back to \eqref{eq:local} and using \eqref{eq:left} and \eqref{eq:right} we achieve
\begin{align}\label{eq:fv_1}
 2\delta\int\limits_{B_r\cap\{0 < u-\delta\phi\}}\U\lap\phi\,dx \leq \L\int\limits_{B_r\cap\{0< u\leq\delta\phi\}}\abs{\grad u}^2dx + o(\delta)\,.
\end{align}
For the last integral on the right hand side we obtain
\begin{align*}
 \int\limits_{B_r\cap\{0< u\leq\delta\phi\}} \abs{\grad u}^2dx&= -\int\limits_{B_r\cap\{0< u\leq\delta\phi\}}u\lap udx +\delta \int\limits_{\partial(B_r\cap\{0< u\leq\delta\phi\})}\phi\partial_\nu u\,dS \\&= o(\delta)
\end{align*}
since $\abs{\{0< u\leq\delta\phi\}}$ tends to zero as $\delta$ tends to zero. Note that we are allowed to integrate by parts since almost every level set $\{u = \delta\phi\}$ is smooth.
Thus,  \eqref{eq:fv_1} simplifies to
\[
 2\delta\int\limits_{B_r\cap\{0 < u-\delta\phi\}}\U\lap\phi\,dx \leq o(\delta). 
\]
Dividing by $\delta$ and letting $\delta$ tend to zero proves the claim.
\end{proof}
With some obvious changes, we can also prove the following corollary.
\begin{corollary}\label{ELI_nonpos}
  Suppose $B_r(x_0)\subset \Omega^-(u)$. Then for each $\phi \in C^\infty_{c}(B_r(x_0))$ with $\phi\geq 0$ there holds
 \[
  \int\limits_{B_r(x_0)\cap\{u<0\}}\U\lap\phi\,dx \geq 0\,.
 \]
\end{corollary}
%%%%
%%%% SECTION 5
\section{The Free Boundary }\label{ch:fb}
The crucial step for all further results concerning the free boundary is the nondegeneracy of a minimizer $u$ along the free boundary
 We roughly follow the idea of C. Bandle and A. Wagner in \cite{BaWa09}. Note that in this paper, the authors chose a penalization term like $\hat{f_\eps}$ as defined in Remark \ref{rem:penal_term}. 
In summary, the crucial point in that proof is the strictly monotone penalization term and the fact that the minimizer satisfies the volume condition. \\
As described in Remark \ref{rem:penal_term}, with a strictly monotone penalization term we would not have been able to show that a minimizer $u$ satisfies the volume condition in our case.  However, we now know that assuming $\eps \leq \eps_0$ each minimizer $\ue$ of $\J_\eps$ satisfies $\abs{\Omega(\ue)}=\omega_0$. Hence, for $\eps\leq\eps_0$ and $\ue$ a minimizer of $\Je$ in $H^{2,2}_0(B)$ there holds
\[
   \L = \L(\Omega(\ue)) = \Je(\ue) = \min\{\L(\Omega) : \Omega \subset B, \; \abs{\Omega}\leq \omega_0 \}.
\]
In the sequel, we show that for $\eps$ sufficiently small the minimizer $\ue$ of $\Je$ is also a minimizer of the functional $\Ie : H^{2,2}_0(B) \to \R$ defined by
\[
   \Ie(v) := \J(v) + \hat{f}_\eps(\abs{\O(v)}).
\]
Thereby, the new term $\hat{f}_\eps$ is defined as in Remark \ref{rem:penal_term}.

\begin{theorem}\label{theo:u_min_I}
 There exists a parameter $\eps_1>0$ such that if $\eps \leq \min\{\eps_0,\eps_1\}$,  then for each $v \in H^{2,2}_0(B)$ there holds
\[
 \mathcal{I}_\eps(u) \leq \mathcal{I}_\eps(v),
\]
where $u = \ue$ is a minimizer of $\Je$.
\end{theorem}
\begin{proof}
Assume that there exists a $v_\eps \in H^{2,2}_0(B)$ with 
\[
   \mathcal{I}_\eps(v_\eps) < \mathcal{I}_\eps(u) = \Lambda.
\]
Note that assuming $\vert\O(v_\eps)\vert\geq\omega_0$ leads to a contradicition since $u$ minimizes $\Je$. Thus, we assume $\vert\O(v_\eps)\vert<\omega_0$ and distinguish  two cases. \\
{\bf a)}\; Let $\vert\O(v_\eps)\vert<\frac{\omega_0}{2} $. 
The homothety of $\Lambda$ implies
\begin{equation}\label{eq:homothety}
  \Lambda(\O(v_\eps)) = \left(\frac{\omega_0}{\vert\O(v_\eps)\vert}\right)^\frac{2}{n}\Lambda\left(\left(\frac{\omega_0}{\vert\O(v_\eps)\vert}\right)^\frac{1}{n}\O(v_\eps)\right) \geq \left(\frac{\omega_0}{\vert\O(v_\eps)\vert}\right)^\frac{2}{n}\Lambda.
\end{equation}
As a consequence, $\mathcal{I}_\eps(v_\eps) < \mathcal{I}_\eps(u) $ implies
\begin{align*}
\left( \left(\frac{\omega_0}{\vert \O(v_\eps)\vert}\right)^\frac{2}{n}-1  \right) \Lambda < \eps(\omega_0 - \vert\O(v_\eps)\vert).
\end{align*}
Due to the assumption on $\vert\O(v_\eps)\vert$ we find
  $\frac{2^\frac{2}{n}-1}{\omega_0}\Lambda < \eps$.
\\
{\bf b)} \; Let $\frac{\omega_0}{2}<\vert\O(v_\eps)\vert < \omega_0$. Then we may write $\vert\O(v_\eps)\vert = \omega_0-\kappa$
with $0<\kappa<\frac{\omega_0}{2}$. In view of \eqref{eq:homothety} $\mathcal{I}_\eps(v_\eps) < \mathcal{I}_\eps(u) $ implies
\begin{align*}
\frac{2}{n\, \omega_0}\Lambda < \eps.
\end{align*}
Thus, the claim follows setting $\eps_1 = \min\left\{\frac{2}{n\, \omega_0}\Lambda, \frac{2^\frac{2}{n}-1}{\omega_0}\Lambda \right\}.$
\end{proof}
The previous theorem shows that, for $\eps \leq \min\{\eps_0, \eps_1\}$, the domain $\Omega(\ue)$ stays optimal even if we reward the Rayleigh quotient $\J$ while concerning a domain with smaller Lebesgue measure. From now on we always consider $\eps  = \min\{\eps_0, \eps_1\}$ and denote $\ue = u$. Thus, every solution $u$ of problem \eqref{P} is a minimizer of $\Ie$, too.\\
The next theorem shows the nondegeneracy of $u$ along the free boundary. 

\begin{theorem}\label{theorem:nondeg}
 There exists a positive constant $c_0$ and a critical radius $r_0 = r_0(n,\omega_0)$ such that for each ball $B_r(x_0)$ with $x_0 \in \partial\Omega(u)$ and $r \leq r_0$ there holds
\[
   c_0\,r \leq \sup_{\overline{B_r(x_0)}}\vert\nabla u\vert.
\] 
The constant $c_0$ is independent of $x_0$ and $r$.
\end{theorem}
\begin{proof}
 We prove the claim by contradiction. Assume that there is a sequence of minimizers $(u_k)_k$, a sequence  $(x_k)_k$ with $x_k\in\partial\Omega(u_k)$ and a sequence of radii $(r_k)_k$ with $r_k\leq r_0$ such that 
\begin{equation}\label{eq:xi_k}
   \xi_k := \frac{1}{r_k}\sup_{\overline{B_{r_k}(x_k)}}\vert\nabla u_k\vert \stackrel{k\to\infty}{\longrightarrow}0. 
\end{equation}
This assumption immediately implies
\[
   \frac{1}{r^2_k}\sup_{\overline{B_{r_k}(x_k)}}\vert u_k\vert \leq 2 \xi_k \stackrel{k\to\infty}{\longrightarrow}0. 
\]
For each $k\in \N$ we choose  $\eta_k \in C^\infty(B_{r_k}(x_k))$ with  $0\leq \eta_k\leq 1$, $\eta_k \equiv 1 $ in $\R^n\setminus B_{r_k}(x_k)$ and $\eta_k \equiv 0$ in $B_\frac{r_k}{2}(x_k)$. 
Then there holds $v_k := \eta_ku_k \in H^{2,2}_0(B)$ and 
\begin{equation}\label{eq:volineq}
  \vert\O(v_k)\vert = \omega_0 - \vert B_\frac{r_k}{2}(x_k)\cap\O_k\vert \geq \omega_0 -\vert B_{r_k}(x_k)\cap\O_k\vert \geq \omega_0 - \vert B_{r_k}(x_k)\vert.
\end{equation}
Recall that we set $\eps:= \min\{\eps_0,\eps_1\}$. Then for each $k\in\N$ Theorem \ref{theo:u_min_I} implies
\[
  \mathcal{I}_\eps(u_k) \leq \mathcal{I}_\eps(v_k).
\]
By definition of $v_k$ and estimate \eqref{eq:volineq} we get the local inequality
\begin{equation}\label{eq:est1}\begin{split}
  &\int\limits_{B_{r_k}(x_k)}\vert\Delta u_k\vert^2\:dx \\ \leq &\int\limits_{B_{r_k}\setminus B_\frac{r_k}{2}(x_k)}\vert\Delta(\eta_ku_k)\vert^2\:dx + \Lambda\int\limits_{B_{r_k}(x_k)}\vert\nabla u_k\vert^2\:dx -\eps\vert B_{r_k}(x_k)\vert \int\limits_B\vert\nabla v_k\vert^2\:dx
\end{split}\end{equation}
We estimate the integrals on the right hand side separately. 
The $C^{1,1}$ regularity of $u_k$ and assumption \eqref{eq:xi_k} imply 
\begin{align*}
  \int\limits_{B_{r_k}\setminus B_\frac{r_k}{2}(x_k)}\vert\Delta (\eta_ku_k)\vert^2\:dx 
&\leq   \int\limits_{B_{r_k}\setminus B_\frac{r_k}{2}(x_k)}\vert\Delta u_k\vert^2\:dx + C(n,\omega_0) \xi_k \vert B_{r_k}(x_k)\vert.
\end{align*}
For the second integral on the right hand side of \eqref{eq:est1} we use assumption \eqref{eq:xi_k} and obtain
\[
  \int\limits_{B_{r_k}(x_k)}\vert\nabla u_k\vert^2\:dx \leq \sup_{\overline{B_{r_k}(x_k)}}\vert \nabla u_k\vert^2\vert B_{r_k}(x_k)\vert \leq \xi_k \vert B_{r_k}(x_k)\vert. 
\]
The last integral in \eqref{eq:est1} can be estimated analougously to estimate \eqref{eq:prob_2} in the proof of Theorem \ref{theorem:vol_satis}. We obtain 
\begin{align*}
  \int\limits_B\vert \nabla v_k\vert^2\:dx&\geq 1 - C(n,\omega_0)r_k^{n+2}
\geq 1 - C(n,\omega_0)r_0^{n+2} \geq \frac{1}{2}.
\end{align*}
The last inequality fixes the critical radius $r_0$. Joining the previous results, we obtain
\[
   0 \leq \int\limits_{B_\frac{r_k}{2}(x_k)}\vert\Delta u_k\vert^2\:dx \leq \left( C(n,\omega_0) \xi_k - \frac{\eps}{2} \right) \vert B_{r_k}(x_k)\vert.
\]
Since we assume $\xi_k$ to converge to zero as $k$ tends to infinity, there exists a $k_0 \in \N$ such that $C(n,\omega_0)\xi_k < \frac{\eps}{2}$ for each $k\leq k_0$.  Thus, for $k = k_0$ we find
\[
  0 \leq \left(C(n,\omega_0) \xi_{k_0}-\frac{\eps}{2}\right) \vert B_{r_{k_0}}(x_{k_0})\vert < 0. 
\]
Obviously, this is contradictory.
\end{proof}

Following the lines of 
\cite{BaWa09}, we derive a positive lower bound on the density of the free boundary. 
\begin{lemma}\label{lem:density}
There exists a positive constant $c_1$ such that for each solution $u$ of the problem~\eqref{P} there holds  
\begin{equation}\label{eq:dens_bound}
 \frac{\abs{B_r(x_0)\cap\Omega(u)}}{\abs{B_r(x_0)}}\geq c_1 \quad \mbox{for all $x_0 \in \partial\Omega(u)$ and $r>0$. }
\end{equation}
\end{lemma}
\begin{proof}
   Let $x_0 \in \partial\Omega(u)$ and $r>0$. Due to Theorem \ref{theorem:nondeg}, there exists an $y \in \overline{B_r(x_0)}\cap\Omega(u)$ such that 
\[
  c_0\,r \leq \sup_{\overline{B_r(x_0)}}\abs{\grad u} = \abs{\grad u(y)}\,.
\]
Now let $\rho \leq r$ be the smallest radius such that $\partial B_\rho(y)\cap\partial\Omega(u)\neq \emptyset$. Hence, there exists an $z \in \partial B_\rho(y)\cap\partial\Omega(u)$ with 
\[
 c_0\,r \leq  \abs{\grad u(y)-\grad u(z)} \leq C(n,\omega_0)\,\rho\,.
\]
This implies immediately 
\[
 \left( \frac{c_0}{C(n,\omega_0)}\right)^n \leq \frac{\rho^n}{(2r)^n} = \frac{\abs{B_\rho(y)}}{\abs{B_{2r}(x_0)}}\leq \frac{\abs{B_{2r}(x_0)\cap\Omega(u)}}{\abs{B_{2r}(x_0)}}.
\]
\end{proof}
 As a direct consequence of the density bound \eqref{eq:dens_bound}, we find that $\partial\Omega(u)$ is a nullset with respect to the $n$-dimensional Lebesgue measure. 
\begin{lemma}\label{lem:nullset}
 Suppose $u\in H^{2,2}_0(B)$ is a solution of the problem \eqref{P}. Then there holds 
\[
  \mathcal{L}^n(\partial\Omega(u)) =0\,.
\]
\end{lemma}
\begin{proof}
  Recall that since $\chi_{\Omega(u)} \in L^1(B)$, almost every $x\in B$ is a Lebesgue Point of $\chi_{\Omega(u)}$. 
Consider a Lebesgue Point $x_0\in \partial\Omega(u)$. Thus, Lemma~\ref{lem:density} implies
\begin{align*}
 0 = \chi_{\Omega(u)}(x_0) &= \lim_{r\to0}\fint\limits_{B_r(x_0)}\chi_{\Omega(u)}\,dx 
= \lim_{r\to0}\frac{\abs{B_r(x_0)\cap\Omega(u)}}{\abs{B_r(x_0)}} \geq c_1>0\,.
\end{align*}
Hence, the free boundary $\partial\Omega(u)$ does not contain any Lebesgue Point of $\chi_{\Omega(u)}$ and therefore $\mathcal{L}^n(\partial\Omega(u))~=~0$.
\end{proof}
In addition, the density estimate \eqref{eq:dens_bound} enables us to derive some more properties of the free boundary.  
The proof of the next lemma follows exaclty as in \cite{BaWa09}. 
\begin{lemma}\label{lem:R_0}
Let $u \in H^{2,2}_0(B)$ be a solution of the problem \eqref{P} and let us assume that the centre of the ball $B$ is contained in $\Omega(u)$. Then  $\partial\Omega(u)\cap\partial B = \emptyset$ if the radius $R_0$ of $B$ is sufficiently large.
\end{lemma}
As a consequence, we may let the radius of $B$ tend to infinity without affecting $\Omega(u)$, provided the centre of $B$ is contained in $\Omega(u)$. In particular, the optimal domain cannot form thin tentacles but remains a bounded domain.
The density estimate \eqref{eq:dens_bound} carries even more information about the shape of the free boundary. 
\begin{remark}\label{remark:corners}
 Let $x_0 \in \partial \Omega(u)$ be the top of a corner on the free boundary.
We denote by $\theta$ the opening angle in $x_0$. Depending on $n$, there holds
\[
 \abs{B_r(x_0)\cap\Omega(u)} = \begin{cases}
  \frac{1}{2}r^2\theta\,, & n=2  \\
 \frac{2}{3} \pi r^3\left(1-\cos\left(\frac{\theta}{2}\right)\right)\,, & n=3
\end{cases}.
\]
Then the lower bound on the density immediately implies
\[
  \theta \geq \begin{cases}
   2\pi c_1 \,, &n=2 \\
 2\arccos (1-2 c_1) \,, &n=3
\end{cases}.
\]
Hence, the opening angle $\theta$ is bounded from below. 
\end{remark}
At this point, we should emphasize that we gained the previous results although we cannot exclude that there are branch points on the free boundary. 

Next, we show the existence of a representative $\W$ of $\U$, which is superharmonic in the nonnegative phase and subharmonic in the nonpositive phase. For this purpose, we need the following definition, which is mainly cited from  \cite{MaZi97}. 
\begin{definition}\label{def:superharm}
 Let $\Omega \subset \R^n$ be an open set. Suppose $w \in L^1_{loc}(\Omega)$. 
Then $w$ is called superharmonic (subharmonic) if $w$ is lower (upper) semicontinuous and 
 \[
  w(x_0) \geq (\leq) \fint\limits_{B_r(x_0)}w(x)\,dx 
 \]
 for each ball $B_r(x_0)\subset\subset \Omega$.
\end{definition}
We now combine ideas of  \cite{CF_79} and  \cite{MaZi97} to gain  the representative $\W$.
\begin{theorem}
  There exists a function $\W: \Omega^+(u)\cup \Omega^-(u) \to\R$ such that 
\begin{itemize}
\item[(a)] $\W = \U$ almost everywhere in $\Omega^+(u)\cup \Omega^-(u)$, 
\item[(b)] $\W$ is superharmonic  in $\Omega^+$ and subharmonic in $\Omega^-$ in the sense of Definition~\ref{def:superharm}\,.
\end{itemize}
\end{theorem} 
\begin{proof}
We restrict ourselves to prove the assertions only in $\Omega^+$. The changes one has to make for proving the other case are obvious.  
Consider $x_0 \in \Omega^+$ and $R>0$ such that $B_R(x_0)\subset \Omega^+$. 
We choose $0<r<s<R$  and set
\[
  \psi_t(x) := \begin{cases}
                      \frac{1}{\omega_n t^n}, & \abs{x-x_0}< t \\
                      0, & \text{otherwise}
                \end{cases}\,.
\]
Following the lines of \cite{MaZi97}, Theorem 2.58,  we construct a sequence of functions $\phi_k \in C^{\infty}_c(B_R(x_0))$  with $\phi_k\geq 0$ and 
\begin{equation}\notag
  \lap \phi_k (x) \stackrel{k\to\infty}{\longrightarrow} \psi_s(\abs{x-x_0}) - \psi_r(\abs{x-x_0}) \; \text{  in  } L^2(B_R(x_0)).
\end{equation}
 For further details in construction the sequence $(\phi_k)_k$ we refer to \cite{MaZi97}. 
Since each $\phi_k$ is a suitable comparison function, Theorem \ref{ELI_nonneg} implies
\begin{align*}
  0\geq \int\limits_{B_R(x_0)\cap\{u>0\}}\U\lap\phi_k\,dx .
\end{align*}
Passing to the limit $k\to\infty$ we obtain
\begin{align*}
\frac{1}{\abs{B_r}}\int\limits_{B_r(x_0)\cap\{u>0\}}\U\,dx &\geq 
\frac{1}{\abs{B_s}}\int\limits_{B_s(x_0)\cap\{u>0\}}\U\,dx\,.
\end{align*}
Hence, for each $x_0 \in \Omega^+$ the function 
\[
\W_r(x_0) := \frac{1}{\abs{B_r}}\int\limits_{B_r(x_0)\cap\{u>0\}}\U\,dx
\]
is nondecreasing as $r$ tends to zero. Furthermore, $\abs{\W_r(x_0)}\leq \norm{\U}_{L^\infty(B)}$ for each $x_0 $ and each $r>0$. Thus, the limit $\lim_{r\to 0}\W_r(x_0)$ exists for every $x_0 \in \Omega^+$ and we set 
\[
\W(x_0) := \lim_{r\searrow 0}\W_r(x_0)\,.
\]
Note that $\W$ is lower semicontinuous and $\W = \U$ almost everywhere in $\Omega^+$.
Since the free boundary is a nullset with respect to the $n$-dimensional Lebesgue measure (see Lemma~\ref{lem:nullset}),  we find
\[
\W(x_0) \geq \fint\limits_{B_r(x_0)}\U \,dx =  \fint\limits_{B_r(x_0)}\W \,dx
\]
and $\W$ is superharmonic in the sense of Definition \ref{def:superharm}. 
\end{proof}
As a consequence of the super-/ subharmonicity of $\W$, $\W$ is positive in each positive nonbranch point and negative in each negative nonbranch point.
\begin{lemma}\label{lem:W_pos_boundary}
 Suppose $x_0 \in \partial\Omega(u)\cap\Omega^+(u)$. Then there holds $\W(x_0)>0$. On the other hand, if $x_0 \in \partial\Omega(u)\cap\Omega^-(u)$, then $\W(x_0)<0$. 
\end{lemma}
\begin{proof}
Again, we restrict ourselves to consider a positive nonbranch point $x_0$. 
We choose $x_0 \in \partial\Omega(u)\cap\Omega^+(u)$. The classical representation formula yields for each ball $B_r(x_0)$ with $B_r(x_0)\subset\Omega^+(u)$
\[
  \int\limits_{B_r(x_0)}\lap u(x)\,\Gamma_r(x-x_0)\,dx = \fint\limits_{\partial B_r(x_0)}u(x)\,dS(x) > 0 
\]
where $\Gamma_r$ is the nonnegative fundamental solution for the Laplacian, which vanishes on $\partial B_r(x_0)$.  Thus,
\[
\psi(r) :=  \int\limits_{B_r(x_0)}\U(x)\Gamma_r(x-x_0)\,dx > \L\int_{B_r(x_0)}u(x)\,\Gamma_r(x-x_0)\,dx \geq 0\,.
\]
Differentiating with respect to $r$ gives us
\[
  \frac{d}{dr}\psi(r)  \leq \frac{r}{n}\W(x_0)
\]
since $\W$ is superharmonic. Integrating with respect to $r$ proves the claim. 
\end{proof}
\begin{corollary}\label{corollary:sign_U}
  There exists an inner neighbourhood $\mathcal{A}^+$ of $\partial\Omega(u)\cap\Omega^+(u)$ and an inner neighbourhood $\mathcal{A}^-$ of $\partial\Omega(u)\cap\Omega^-(u)$ such that $\W(x)>0$ for every $x\in\mathcal{A}^+$ and $\W(x)<0$ for every $x \in \mathcal{A}^-$.
\end{corollary}
\begin{remark}
If we could establish a lower bound for $\abs{\lap u}$ in an inner neighbourhood of $\partial\Omega$, the previous corollary would imply that $u$ cannot change its sign close to the free boundary. In case of a smooth optimal domain ($\partial\Omega$ at least $C^{2,\alpha}$), the first domain variation of the optimal domain would imply $\lap u = const.$ in $\partial\Omega$. This condition would suffice to disprove the existence of branch points. Thus, future works should examine the regularity of the free boundary $\partial\Omega$.
\end{remark}
\subsubsection*{Acknowledgements}
The author likes to thank Alfred Wagner for many fruitful discussions. 

\subsubsection*{Published in Calculus of Variations and Partial Differential Equations}
The final publication will be available at link.springer.com.
%%%%%%%%%%%%%%%%%%%%%%%%%%%%%%%%%%%%%%%%%%%%%%%%%%%%%%%%%%%%%%%%%%%%%%%%%%%%%%%%%%%%%%%%%%%%%%%%%%%%%%

\textsc{RWTH Aachen University, Institut für Mathematik, Templergraben 55, D-52062 Aachen, Germany}\\
\textit{E-mail adress}: \texttt{stollenwerk@instmath.rwth-aachen.de}
\end{document}